\documentclass[a4paper,12pt]{article}

\usepackage[english]{babel}
\usepackage{amsmath, amssymb, amsfonts,amsthm}
\usepackage{latexsym}
\usepackage{graphicx}
\usepackage[ansinew]{inputenc}
\usepackage[inner=2.5cm,outer=2.5cm,top=2cm,bottom=2cm, includehead,includefoot]{geometry}
\usepackage{dsfont}
\usepackage[ansinew]{inputenc}
\newtheorem{theorem}{Theorem}[section]
\newtheorem{lemma}[theorem]{Lemma}
\newtheorem{corollar}[theorem]{Corollary}
\newtheorem{prop}[theorem]{Proposition}
\theoremstyle{remark}

\theoremstyle{definition}
\newtheorem{example}[theorem]{Example}
\newtheorem{defi}[theorem]{Definition}
\newcommand{\indi}{\mathds{1}} 

\newcommand{\ind}{\perp \! \! \! \perp}

\newcommand{\td}{\mathbb{T}^d}
\newcommand{\T}{\mathbb{T}}
\newcommand{\D}{\mathbb{D}}
\providecommand{\abs}[1]{\lvert#1\rvert}

\providecommand{\norm}[1]{\lVert#1\rVert}
\newcommand{\tbf}{\textbf{t}}
\newcommand{\sbf}{\textbf{s}}
\newcommand{\nbf}{\textbf{n}}
\newcommand{\hbf}{\textbf{h}}

\newcommand{\kbf}{\textbf{k}}
\newcommand{\mbf}{\textbf{m}}

\newcommand{\Bbf}{\textbf{B}}
\newcommand{\Nbf}{\textbf{N}}

\newcommand{\zbf}{\textbf{z}}

\newcommand{\ew}{\mathbb{E}}
\newcommand{\var}{\mbox{\textbf{Var}}}

\newcommand{\F}{\mathcal{F}} 
 
\newcommand{\R}{\mathds{R}}
\newcommand{\PP}{\mathbb{P}}
\newcommand{\N}{\mathds{N}}
\newcommand{\C}{\mathbb{C}}
\newcommand{\Z}{\mathbb{Z}}

\newcommand{\bs}{\begin{sffamily}}
\newcommand{\es}{\end{sffamily}}
\newcommand\be{\begin{eqnarray}}
\newcommand\ee{\end{eqnarray}}
\newcommand\bee{\begin{eqnarray*}}
\newcommand\eee{\end{eqnarray*}}
\newcommand\bi{\begin{itemize}}
\newcommand\ei{\end{itemize}}

\renewenvironment{proof}{\vspace{0.1ex}\noindent{\it Proof.}\hspace{0.1em}}
	{\hfill$\Box$}

\begin{document}
 \pagenumbering{arabic}

\title{Strictly stationary solutions of spatial ARMA equations}
\author{Martin Drapatz\footnote{Institut für Mathematische Stochastik, TU Braunschweig, Pockelsstraße 14, D-38106 Braunschweig, Germany m.drapatz@tu-bs.de}}
\maketitle


\begin{abstract}
The  generalization of the ARMA time series model to the multidimensional index set $\Z^d$, $d\ge2$, is called spatial ARMA model. The purpose of the following is to specify necessary conditions and sufficient conditions for the existence of strictly stationary solutions of the ARMA equations when the driving noise is i.i.d. Two different classes of strictly stationary solutions are studied, solutions of causal and non-causal models. For the special case of a first order model on $\Z^2$ conditions are obtained, which are simultaneously necessary and sufficient.
\end{abstract}
{\bfseries Keywords:} causality, random fields, spatial ARMA model, strict stationarity. 
\section{Introduction}

Let $d\in\N$, usually $d>1$, and $(Y_\tbf)_{\tbf=(t_1,\ldots,t_d)\in\Z^d}$ a $d$-dimensional complex-valued random field living on a probability space $(\Omega,\F,\PP)$. If $(Y_\tbf)_{\tbf\in\Z^d}$ fulfills the equations
\begin{eqnarray}
Y_{\textbf{t}} - \sum_{\textbf{n}\in R} \phi_{\textbf{n}}Y_{\textbf{t}-\textbf{n}} = Z_{\textbf{t}} + \sum_{\textbf{n}\in S} \theta_{\textbf{n}} Z_{\textbf{t}-\textbf{n}},\quad \tbf\in\Z^d,\label{eq1}
\end{eqnarray}
where $(\phi_\nbf)_{\nbf\in R},(\theta_\nbf)_{\nbf\in S}\subset \C$, $R$ and $S$ are finite subsets of $\N_0^d\backslash \{\boldsymbol 0\}$ or more generally of $\Z^d\backslash \{\boldsymbol 0\}$, and $(Z_\tbf)_{\tbf\in\Z^d}$ is an i.i.d. complex-valued random field on $(\Omega,\F,\PP)$, we call $(Y_\tbf)_{\tbf\in\Z^d}$ an \emph{ARMA random field}, where ARMA is short for autoregressive moving average. The spatial ARMA model defined by \eqref{eq1} is a natural generalization of the well-known ARMA time series model (see e.g. Brockwell and Davis \cite{Brockwelldavis}, Chapter 3) to higher dimensional index sets $\Z^d$, $d>1$. The spatial ARMA model was considered long ago by Whittle \cite{Whittle} and many others (e.g. \cite{Tjostheim}, \cite{Besag}, \cite{Basu}) had been working on this topic. However, most work has been spent on weakly stationary solutions of the spatial ARMA model and their statistics. \\
\indent For the time series model ($d=1$), Brockwell and Lindner \cite{Brockwelllindner} obtained necessary and sufficient conditions for the existence of strictly stationary solutions of \eqref{eq1}. In this article we generalize those results and obtain some necessary and some sufficient conditions for the existence of strictly stationary solutions of \eqref{eq1}, in terms of some moment conditions on the white noise $(Z_\tbf)_{\tbf\in\Z^d}$ and zero sets of the characteristic polynomials 
\begin{eqnarray*}
\Phi(\zbf)&=&1- \sum_{\textbf{n}\in R}\phi_\nbf \zbf^{\nbf}, \quad \text{and}\\
\Theta(\zbf)&=&1+ \sum_{\textbf{n}\in S}\theta_\nbf \zbf^{\nbf},\quad \zbf=(z_1,\ldots,z_d)\in\C^d,
\end{eqnarray*}
corresponding to the recurrence equation \eqref{eq1}. The polynomial $\Phi$ is called \emph{autoregressive polynomial} and $\Theta$ \emph{moving average polynomial} (we speak of polynomials even if $R,S\subset \Z^d$). It is known that a sufficient condition for the existence of a weakly stationary solution, when usually $(Z_\tbf)_{\tbf\in\Z^d}$ is considered to be only uncorrelated white noise with mean zero, is given by (see Rosenblatt \cite{Rosenblatt}, page 60)
\begin{eqnarray}
\int_{\td}\left|\frac{\Theta(e^{-i\tbf})}{\Phi(e^{-i\tbf})}\right|^2 d\lambda^d(\tbf)<\infty\label{eq2},
\end{eqnarray}
where $\td$ is the $d$-fold cartesian product of the factor space $\mathbb{T}=\R/2\pi\Z$, which we identify by $(-\pi,\pi]$, and $\lambda^d$ is the Lebesgue measure on $\R^d$ limited to $\td$. By spectral density arguments it can easily be shown that this condition is also necessary. Condition \eqref{eq2} will also play a decisive role, when strictly stationary solutions are considered. 
\par There are several differences between $d=1$ and higher dimensional models with $d>1$, which bring some difficulties: first of all, polynomials can not be factored completely like in one dimension, which implies that a quotient of polynomials in several variables may have common zeros that cannot be canceled out. Another difference is that even though $\Phi(e^{-i\cdot})$ may have zeros on $\td$, it is possible that \eqref{eq2} holds, even if $\Theta(\zbf)\equiv 1$. Furthermore, we have to deal with multiple sums $\sum_{\kbf\in\N_0^d} X_\kbf$ for some random field $(X_\kbf)_{\kbf\in\Z^d}$, which do not necessarily converge absolutely. Therefore a type of convergence defined by Klesov \cite{Klesov}, namely almost sure convergence in the rectangular sense, will be used.  \\
\indent The article is structured as follows: in Section 2 we study linear strictly stationary solutions. After that in Section 3 we go on considering strictly stationary causal solutions, without assuming them a priori to be linear. Then in Section 4 a full characterization of necessary and sufficient conditions for the existence of strictly stationary causal solutions of a first order autoregressive model in dimension two will be given. \\
The following notation will be used: vector-valued variables will be printed bold and the multi-index notation 
$$\zbf^\nbf=z_1^{n_1}\cdots z_d^{n_d},\quad \zbf=(z_1,\ldots, z_d)\in\C^d,\ \nbf=(n_1,\ldots, n_d)\in\Z^d, \ e^{i\tbf}=(e^{it_1},\ldots,e^{it_d}),\ \tbf\in\td,$$
will be applied. To indicate that two random variables $X$ and $Y$ are independent, the symbol $X \ind Y$ will be used. Furthermore the Backward Shift Operator $\Bbf=(B_1,\ldots, B_d)$, where $B_i$ shifts the $i$th coordinate back by one, i.e. for the $i$th unit vector $e_i$ in $\R^d$ we have
$$B_i Z_\tbf= Z_{\tbf-e_i},\quad i=1,\ldots,d,\quad \tbf\in\Z^d,$$
will be used to write the ARMA equation \eqref{eq1} in a compact form as
$$\Phi(\Bbf)Y_\tbf=\Theta(\Bbf)Z_\tbf,\quad \tbf\in\Z^d.$$
The Hilbert space of functions $f:\td\to\C$, which are square integrable with respect to $\lambda^d$, will be denoted by $L^2(\td)$. If condition \eqref{eq2} is fulfilled, the existence of a Fourier expansion
\begin{equation}
\frac{\Theta(e^{-i\tbf})}{\Phi(e^{-i\tbf})}=\sum_{\kbf\in\Z^d}\psi_\kbf e^{-i\kbf\tbf},\quad (\psi_\kbf)_{\kbf\in\Z^d}\subset\C,\quad \tbf\in\td,\label{end1}
\end{equation}
where $\kbf\tbf=\kbf\cdot\tbf=\sum_{i=1}^d k_i t_i$ denotes the Euclidean inner product on $\R^d$, is assured, see Shapiro \cite{Shapiro}, Theorem 2.2. Plugging \eqref{end1} into \eqref{eq2}, it is easy to see that the coefficients $(\psi_\kbf)_{\kbf\in\Z^d}$ are square summable, i.e. $\sum_{\kbf\in\Z^d}\abs{\psi_\kbf}^2<\infty$. By $H^2$ we denote the Banach space containing all functions $f:\D^d\to\C$ holomorphic on the open unit polydisc $\D^d=\{\zbf=(z_1,\ldots,z_d)\in\C^d: \abs{z_i}<1, \ i=1,\ldots,d\}$ and fulfilling
$$\norm{f}^2_{H^2}:=\sup_{0\le r<1} \frac{1}{(2\pi)^d}\int_{\td}\left|f(r e^{-i\tbf})\right|^2 d\lambda^d(\tbf)<\infty.$$

If a function $f:\D^d\to\C$ is holomorphic, it admits a power series expansion $f(\zbf)=\sum_{\kbf\in\N_0^d}a_\kbf \zbf^\kbf$, see Range \cite{Range}, Theorem 1.6. Thus, a function $f:\D^d\to\C\in H^2$ admits a representation $f(\zbf)=\sum_{\kbf\in\N_0^d}a_\kbf \zbf^\kbf$ and  $f\in H^2$, if and only if $\sum_{\kbf\in\N_0^d}\abs{a_\kbf}^2<\infty$. To see that, notice that
$$\norm{f}^2_{H^2}=\sup_{0\le r<1}\sum_{\kbf\in\N_0^d}\abs{a_\kbf}^2r^{2\abs{\kbf}}.$$
Hence, each function $f\in H^2$ can be identified with its \glqq boundary function\grqq \ $g:\td\to\C$, whose Fourier expansion is given by $g(e^{-i\tbf})=\sum_{\kbf\in\N_0^d} a_\kbf e^{-i\kbf\tbf}$. Further
$$\norm{f}^2_{H^2}=\sup_{0\le r<1} \sum_{\kbf\in\N_0^d}\abs{a_\kbf}^2r^{2\abs{\kbf}}=\sum_{\kbf\in\N_0^d}\abs{a_\kbf}^2=\frac{1}{(2\pi)^d} \int_{\td} \abs{g(e^{-i\tbf})}^2 d\lambda^d(\tbf)=:\norm{g}^2_{L^2(\td)},$$
so that $H^2$ can be identified with a closed subspace of $L^2(\td)$, more precisely, the space of all functions $g\in L^2(\td)$, whose Fourier coefficients $(a_\kbf)_{\kbf\in\Z^d}$ vanish for $\kbf\in\Z^d\backslash \N_0^d$. The space $H^2$ is called {\it Hardy space}. For more details about Fourier Analysis and Hardy spaces in several variables see Shapiro \cite{Shapiro} or Rudin \cite{Rudin}. Beside Fourier expansions, Laurent expansions in several variables will be utilized. All results from function theory in several variables used in this work can be found in Range \cite{Range}. 


\section{Linear Strictly Stationary Solutions}
\setcounter{equation}{0}
In this section we introduce the notion of {\it linear strictly stationary ARMA random fields} and establish necessary and sufficient conditions for the existence of solutions of the ARMA equations for this class of random fields. In the whole section we assume that $R$ and $S$ are subsets of $\Z^d\backslash \{\textbf{0}\}$.
\begin{defi} A random field $(Y_\tbf)_{\tbf\in\Z^d}$, which solves the ARMA equation \eqref{eq1} where $(Z_\tbf)_{\tbf\in\Z^d}$ is an i.i.d. noise, is called {\it linear strictly stationary solution}, if there are coefficients $(\psi_\kbf )_{\kbf\in\Z^d}\subset\C$, such that
\begin{equation*}
Y_\tbf=\sum_{\kbf\in\Z^d} \psi_\kbf Z_{\tbf-\kbf},\quad \tbf\in\Z^d,
\end{equation*}
where the right-hand side converges almost surely absolutely.
\end{defi}
Obviously, a linear strictly stationary solution is indeed strictly stationary.

\begin{theorem}
Let $R$ and $S$ be subsets of $\Z^d\backslash\{\boldsymbol 0\}$ and $(Z_\tbf)_{\tbf\in\Z^d}$ an i.i.d. random field. The ARMA equation \eqref{eq1} admits a linear strictly stationary solution if and only if 
\begin{eqnarray*}
&&\frac{\Theta(e^{-i\cdot})}{\Phi(e^{-i\cdot})}\in L^2(\td),
\end{eqnarray*}
and if
\begin{equation}
Y_\tbf=\sum_{\kbf\in\Z^d}\psi_\kbf Z_{\tbf-\kbf},\quad \tbf\in\Z^d,\label{eq3a}
\end{equation}
converges almost surely absolutely, where
\begin{equation}
\frac{\Theta(e^{-i\tbf})}{\Phi(e^{-i\tbf})}=\sum_{\kbf\in\Z^d}\psi_\kbf e^{-i\kbf\tbf},\quad \tbf\in\td,\label{eq3b}
\end{equation}
denotes the Fourier expansion of $\Theta(e^{-i\cdot})/\Phi(e^{-i\cdot})$. If these two conditions are satisfied, then a linear strictly stationary solution is given by (\ref{eq3a}).\bigbreak \label{thm1}
\end{theorem}

\begin{proof}
Suppose both conditions are fulfilled. Applying the operator $\Phi(\Bbf)$ on $(Y_\tbf)_{\tbf\in\Z^d}$ as defined in \eqref{eq3a} yields
\begin{align*}
\Phi(\Bbf)Y_\tbf=&Y_\tbf-\sum_{\nbf\in R}\phi_\nbf Y_{\tbf-\nbf}= \sum_{\kbf\in\Z^d} (\underbrace{\psi_\kbf-\sum_{\nbf\in R} \phi_\nbf \psi_{\kbf-\nbf}}_{=:\xi_\kbf}) Z_{\tbf-\kbf}.
\end{align*}
The random field  $(Y_\tbf)_{\tbf\in\Z^d}$ solves the ARMA equations, if the coefficients $(\xi_\kbf)_{\kbf\in\Z^d}$ fulfill
\begin{align}
\xi_\kbf:=\psi_\kbf-\sum_{\nbf\in R}\phi_\nbf \psi_{\kbf-\nbf}=\left\{
\begin{aligned}
 \theta_\kbf\quad  , \quad& \kbf\in S \backslash \{\textbf{0}\},\\
 1\quad ,\quad & \kbf=\textbf{0}, \\
 0 \quad , \quad&\text{otherwise.} \end{aligned}\right. \label{koefff}
\end{align}

To prove the validity of these equalities we compare the coefficients $(\xi_\kbf)_{\kbf\in\Z^d}$ with those of the corresponding Fourier series. Multiplying both sides of equation \eqref{eq3b} by $\Phi(e^{-i\tbf})$ yields
\begin{align}
\Theta(e^{-i\tbf})=1+ \sum_{\nbf\in S}\theta_\nbf  e^{-i\nbf \tbf}=\Phi(e^{-i\tbf}) \left(\sum_{\kbf\in\Z^d} \psi_\kbf e^{-i\kbf\tbf} \right)=\sum_{\kbf\in\Z^d} \left(\psi_\kbf -\sum_{\nbf\in R}\phi_\nbf \psi_{\kbf-\nbf} \right)  e^{-i\kbf\tbf}.\label{new}
\end{align}
Comparing the coefficients in equation \eqref{new}, the validity of \eqref{koefff} is obtained, which completes the proof of sufficiency. \bigbreak
Suppose the random field $(Y_\tbf)_{\tbf\in\Z^d}$ is a linear strictly stationary solution of the ARMA equations. Thus, it has a representation $Y_\tbf=\sum_{\kbf\in\Z^d}\psi_\kbf Z_{\tbf-\kbf}$ for $\tbf\in\Z^d$ for some sequence $(\psi_\kbf)_{\kbf\in\Z^d}\subset\C$,
where the right-hand side converges almost surely absolutely. By an application of Theorem 5.1.4 of Chow and Teicher \cite{Chow} this implies the square summability of the coefficients $(\psi_\kbf)_{\kbf\in\Z^d}$. The Theorem of Riesz-Fischer (see Stein and Weiss \cite{Stein}, Theorem 1.7) now implies that there exists a function $\Psi(e^{-i\cdot})$ in $L^2(\td)$, whose Fourier coefficients are precisely $(\psi_\kbf)_{\kbf\in\Z^d}$. Hence the Fourier expansion of $\Psi(e^{-i\cdot})$ is given by $\Psi(e^{-i\tbf})=\sum_{\kbf\in\Z^d} \psi_\kbf e^{-i\kbf\tbf}$ for $\tbf\in\td$.
Yet again, we can compare the coefficients of the ARMA equation $\Phi(\Bbf)Y_\tbf=\Theta(\Bbf)Z_\tbf$ and those of the product $\Phi(e^{-i\cdot})\Psi(e^{-i\cdot})$ and conclude
\begin{equation}
\Phi(e^{-i\tbf})\Psi(e^{-i\tbf})=\Theta(e^{-i\tbf}),\quad \tbf\in\td.\label{added1}
\end{equation}
We define the measurable set $N:=\{\tbf\in\td: \Phi(e^{-i\tbf})=0	\}$ and obtain by equation \eqref{added1}
\begin{align*}
\int_{\td\backslash N}\left|\frac{\Theta(e^{-i\tbf})}{\Phi(e^{-i\tbf})}\right|^2 d\lambda^d(\tbf)=\int_{\td\backslash N}\left|\Psi(e^{-i\tbf})\right|^2 d\lambda^d(\tbf)<\infty.\label{wqw}
\end{align*}
Furthermore, by Theorem 3.7 of Range \cite{Range}, the set $N$ is a $\lambda^d-$nullset. Thus, $\Theta(e^{-i\cdot})/\Phi(e^{-i\cdot})\in L^2(\td)$ and because of the uniqueness of the Fourier expansion, the Fourier coefficients of  $\Theta(e^{-i\cdot})/\Phi(e^{-i\cdot})$ are given by $(\psi_\kbf)_{\kbf\in\Z^d}$.
\end{proof}\\

An immediate question is under which conditions the right-hand side of equation \eqref{eq3a} converges almost surely absolutely. Before giving a sufficient condition in Proposition \ref{prop111}, we need the following lemma.

\begin{lemma}
For $n,d\in\N$ denote the cardinality of the set $\left\{	\kbf\in\Z^d: \abs{k_1}+\ldots+\abs{k_d}=n		\right\}$ by $h_d(n)$. Then $h_d(n)$ can be estimated from above by $C_d n^{d-1}$ for some constant $C_d>0$.
\label{lem1}
\end{lemma}

\begin{proof}
For $d=1$ we have $h_1(n)=\left|\left\{	k\in\Z: \abs{k}=n		\right\}\right|=2$.
Suppose the assumption is valid for $d\in\N$. Then the following identity holds for $d+1$
\begin{equation*}
\left\{	\kbf\in\Z^{d+1}: \abs{k_1}+\ldots+\abs{k_{d+1}}=n		\right\}=\bigcup_{k=0}^n \{\kbf\in\Z^{d+1}: \sum_{i=1}^d\abs{k_i}=n-k, \abs{k_{d+1}}=k\}.
\end{equation*}
Each set of this union with $k\neq n$ has cardinality less than or equal to $2C_d(n-k)^{d-1}$ by assumption and for $k=n$ the cardinality is two. Thus we can conclude $h_{d+1}(n)\le 2C_d n^{d}+2\le C_{d+1} n^d$ for some constant $C_{d+1}>0$.
\end{proof}\\

We define for $z\ge0$ the postive part of the natural logarithm as $\log_+(z):=\max(\log(z), 0)$. In the following proposition sufficient conditions for the existence of a linear strictly stationary solution are given.
\begin{prop}
Let $R$ and $S$ be subsets of $\Z^d\backslash\{\boldsymbol 0\}$. If the autoregressive polynomial $\Phi(e^{-i\cdot})$ has no zero on $\td$, 
then for appropriate $\boldsymbol r=(r_1,\ldots,r_d), \boldsymbol \rho=(\rho_1,\ldots,\rho_d)$, $0\le r_i <1<\rho_i$, $i=1\ldots,d$, a Laurent expansion of $\Theta(e^{-i\cdot})/\Phi(e^{-i\cdot})$ exists given by
\begin{equation*}
\frac{\Theta(\zbf)}{\Phi(\zbf)}=\sum_{\kbf\in\Z^d}\psi_\kbf \zbf^\kbf ,\quad \zbf\in K(\textbf{r},\boldsymbol \rho):=\{\zbf=(z_1,\ldots,z_d)\in\C^d: r_i<\abs{z_i}<\rho_i, \ i=1,\ldots,d\}.
\end{equation*}
If further
\begin{equation*}
\ew \log^d_+|Z_1|<\infty,
\end{equation*}
then
\begin{equation*}
Y_t:=\sum_{\kbf\in\Z^d}\psi_\kbf Z_{\tbf-\kbf},\quad \tbf\in\Z^d,
\end{equation*}
converges almost surely absolutely. In particular, the random field $(Y_\tbf)_{\tbf\in\Z^d}$ solves the ARMA equation \eqref{eq1}.
\label{prop111}
\end{prop}
\begin{proof}
If the autoregressive polynomial $\Phi(e^{-i\cdot})$ does not possess any zeros on $\td$, then the quotient $\Theta(\zbf)/\Phi(\zbf)$ is holomorphic in $K(\textbf{r},\boldsymbol \rho)$, where $0\le r_i<1<\rho_i$ for suitable $\textbf{r}=(r_1,\ldots,r_d), \boldsymbol \rho=(\rho_1,\ldots,\rho_d)$. Furthermore Proposition 1.4 in Range \cite{Range} assures the existence of a Laurent expansion
\begin{equation*}
\Theta(\zbf)/\Phi(\zbf)=\sum_{\kbf\in\Z^d}\psi_\kbf \zbf^\kbf ,\quad \zbf\in K(\textbf{r},\boldsymbol \rho), \quad 0\le \textbf{r}<1<\boldsymbol \rho,
\end{equation*}
and the validity of the Cauchy estimates in $d$ variables, i.e. there are constants $M, c > 0$ such that
\begin{equation*}
|\psi_\kbf|\le M e^{-c(\abs{k_1}+\ldots+\abs{k_d})},  \quad \forall \kbf\in\Z^d.
\end{equation*}
For $n\in\N$ the number of possibilities of $\kbf\in\Z^d$ fulfilling $|k_1|+\ldots+|k_d|=n$ can be estimated from above by $C_d n^{d-1}$, $C_d>0$, see Lemma \ref{lem1}. Thus, for $c'\in(0,c)$ it follows
\begin{align}
\sum_{\kbf\in\Z^d} \PP(\ | \psi_\kbf Z_{\tbf-\kbf} |>e^{-c' (|k_1|+\ldots+|k_d|) }\ ) &\le\sum_{\kbf\in\Z^d}  \PP(M | Z_{\tbf-\kbf}|>e^{(c-c') (|k_1|+\ldots+|k_d|)} ) \notag \\
&= \sum_{\kbf\in\Z^d}  \PP(\log_+(M |Z_{\textbf{0}}|)>(c-c') (|k_1|+\ldots+|k_d|) ) \label{reihe}\notag \\
&\le 1+ C_d \sum_{n=1}^{\infty} n^{d-1}  \ \PP(\log_+(M |Z_{\textbf{0}}|)> n  (c-c')  ). 
\end{align}
We define the random variable $X=\log_+(M |Z_{\textbf{0}}|)/(c-c')$. 
Using the two inequalities 
\begin{align*}
\PP(X>n)\le \PP(X>x)\qquad  &\mbox{for} \ x\in(n-1,n], \quad n\in\N, \\
n \le 2x \qquad  &\mbox{for} \ x\in(n-1,n], \quad n\in\N \backslash \{1\}, 
\end{align*}
the last series in equation (\ref{reihe}) can be estimated from above as follows
\begin{align}
\sum_{n=2}^{\infty} n^{d-1} \ \PP(X >  n  )  \le& \sum_{n=2}^{\infty} \int_{n-1}^{n} \PP(X>x) \ (2x)^{d-1} \ dx  \notag \\
=& \ 2^{d-1}\int_1^{\infty} \PP(X>x) \ x^{d-1} \ dx \notag \\
\le& \ \frac{2^{d-1}}{d} \ \ew(X^d)=  \frac{2^{d-1}}{d} \ew\left(\frac{\log_+(M |Z_{\textbf{0}}|)}{(c-c')}\right)^d\label{ewert}< \infty,
\end{align}
where the last inequality is valid, because the expected value in (\ref{ewert}) is finite, if and only if  $\ew \log_+^d |Z_{\textbf{0}}|<\infty$. 
Applying the Borel Cantelli Lemma implies that the event
\begin{equation*}
 \{ \ |\psi_\kbf Z_{\tbf-\kbf}| > e^{-c' (|k_1|+\ldots+|k_d|)} \	\mbox{for infinitely many} \ \kbf \in \Z^d \}
 \end{equation*}
has probability zero. The almost sure majorant $\sum_{\kbf\in\Z^d} e^{-c' (|k_1|+\ldots+|k_d|)}$ is absolutely convergent, and hence the series $Y_\tbf=\sum_{\kbf\in\Z^d}\psi_\kbf Z_{\tbf-\kbf}$ converges almost surely absolutely for all $\tbf\in\Z^d$.
\end{proof}\\

For $d=1$ the condition $\Phi(z)\neq0$ for all $z\in\C, \abs{z}=1$, is a necessary and sufficient condition for the uniqueness of a strictly stationary solution, provided one exists, see Brockwell and Lindner \cite{Brockwelllindner}. For $d>1$ we do not know whether the analog condition $\Phi(e^{-i\tbf})\neq 0$ for all $\tbf\in\td$ is sufficient for the uniqueness of linear strictly stationary solutions. However, the necessity of this condition is shown in the following lemma.

\begin{lemma}
Suppose $(Y_\tbf)_{\tbf\in\Z^d}$ is a strictly stationary solution of \eqref{eq1}. Suppose further that $\Phi(e^{-i\cdot})$ has a zero $\boldsymbol\lambda\in\td$. Finally, suppose the underlying probability space is rich enough to support a random variable $U$, which is independent of $(Y_\tbf)_{\tbf\in\Z^d}$ and uniformly distributed on $[0,1]$. Then $(Y_\tbf+ e^{i2\pi U}e^{i\tbf\boldsymbol\lambda})_{\tbf\in\Z^d}$ is another strictly stationary solution of \eqref{eq1}. In particular, the strictly stationary solution of \eqref{eq1} is not unique.
\end{lemma}
\begin{proof}
It is easy to see that the random field
\begin{equation*}
W_\tbf= e^{i2\pi U}e^{i\tbf\boldsymbol\lambda},\quad \tbf\in\Z^d,
\end{equation*}
is strictly stationary. Because of the independence of $U$ and $(Y_\tbf)_{\tbf\in\Z^d}$, the random field $(X_\tbf)_{\tbf\in\Z^d}$ defined by $X_\tbf=Y_\tbf+W_\tbf$, is also strictly stationary. Furthermore we have
\begin{equation*}
\Phi(\Bbf)W_\tbf=W_\tbf-\sum_{\nbf\in R} \phi_\nbf W_{\tbf-\nbf}=e^{i\tbf\boldsymbol\lambda}e^{i2\pi U} \Phi(e^{-i\boldsymbol\lambda})=0,
\end{equation*}
which shows that $(X_\tbf)_{\tbf\in\Z^d}$ is another solution of the ARMA equation \eqref{eq1}.
\end{proof}\bigbreak
Notice that, if the conditions of the preceding lemma are fulfilled and $(Y_\tbf)_{\tbf\in\Z^d}$ is a linear strictly stationary solution, then the random field $(Y_\tbf+ e^{i2\pi U}e^{i\tbf\boldsymbol\lambda})_{\tbf\in\Z^d}$ is an example for a strictly stationary solution which is not linear.


\section{Causal solutions}\setcounter{equation}{0}
In this section we study necessary conditions and sufficient conditions for the existence of causal solutions. We define for $\tbf=(t_1,\ldots,t_d), \sbf=(s_1,\ldots,s_d)\in\Z^d$ the index sets $\{\sbf\le\tbf\}$ induced by the relation \glqq$\le$\grqq \ on $\Z^d$:
$$\{\sbf\le\tbf\}:=\{\sbf\in\Z^d: \ s_i\le t_i,\ i=1,\ldots,d\}.$$
\begin{defi} A strictly stationary random field $(Y_{\tbf})_{\tbf\in\Z^d}$, which fulfills the ARMA equation \eqref{eq1}, is called {\it causal solution} of the spatial ARMA model, if $Y_\tbf$ is measurable with respect to $\sigma(Z_\sbf: \sbf\le \tbf)$ for each $\tbf\in\Z^d$.\label{defii}
\end{defi}
When considering causal solutions, it makes sense to restrict the index sets $R,S$ in equation \eqref{eq1} to subsets of $\N_0^d\backslash\{\boldsymbol 0\}$. We assume this throughout the whole section. We will see that the symmetrization $(\tilde Y_{\tbf})_{\tbf\in\Z^d}$ of a causal solution $(Y_{\tbf})_{\tbf\in\Z^d}$ admits a linear representation $\tilde Y_\tbf=\sum_{\kbf\in\N_0^d}\psi_\kbf \tilde Z_{\tbf-\kbf}$ for some coefficients $(\psi_\kbf)_{\kbf\in\N_0^d}\subset\C$ and a symmetrization $(\tilde Z_{\tbf})_{\tbf\in\Z^d}$ of $(Z_{\tbf})_{\tbf\in\Z^d}$.
But in contrast to the definition of linear strictly stationary solutions, a specific type of convergence is not required by Definition \ref{defii}. However, implicitly this sum has to convergence {\it almost surely in the rectangular sense}, as we will see later on in Theorem \ref{theorem2}.
\begin{defi}[Klesov \cite{Klesov}, Definitions 1 and 3]
Let $(Z_\kbf)_{\kbf\in\N_0^d}$ be a real-valued random field. The multiple series $\sum_{\kbf\in\N_0^d} Z_{\kbf}$ {\it converges almost surely in the rectangular sense}, if the limits
$$\lim_{k\to\infty} \sum_{k_1=0}^{N_{1k}}\cdots \sum_{k_d=0}^{N_{dk}} Z_\kbf,$$
for all sequences $\left(N_{1k},\ldots,N_{dk}\right)_{k\in\N}\subset\N_0^d$ with $\min(N_{1k},\ldots,N_{dk})\to\infty \ (k\to\infty)$ almost surely exist and coincide. 
\end{defi}
For this mode of convergence a generalization to multiple series of the three series theorem of Kolmogorov is valid. Precisely, the following theorem holds.
\begin{theorem}[Klesov \cite{Klesov2} and Klesov \cite{Klesov}, Theorem C]
Let $(X_\nbf)_{\nbf\in\N_0^d}$ be a real-valued random field of independent random variables and define $X_\nbf^c:= X_\nbf \indi_{\{\abs{X_\nbf}<c\}}$. Then almost sure convergence of 
\begin{equation}
\sum_{\nbf\in\N_0^d}X_\nbf,\label{new11}
\end{equation}
in the rectangular sense and the condition
\begin{equation}
\PP\left(\abs{X_{\nbf_k}}>\epsilon\right)\to 0 \quad (k\to\infty),\quad\forall \epsilon>0, \label{redu}
\end{equation}
for all sequences $(\nbf_k)_{k\in\N}$ with $\max(n_{1k},\ldots,n_{dk})\to\infty$, are equivalent to the convergence of the following three series for some $c>0$, and hence for any $c>0$: 
\begin{align*}
&\sum_{\nbf\in\N_0^d} \PP(\abs{X_\nbf}\ge c)<\infty\tag{A},\\
&\sum_{\nbf\in\N_0^d} \ew(X^c_\nbf) \tag{B},\\
&\sum_{\nbf\in\N_0^d} \var(X^c_\nbf)\tag{C}.
\end{align*}
Here, convergence of $(B)$ is to be understood as (almost surely) rectangular.\label{theoremklesov2}
\end{theorem}

If $d=1$ and $\sum_{n=1}^\infty X_n$ converges almost surely, condition \eqref{redu} is fulfilled automatically. For $d>1$ the following example from Klesov \cite{Klesov} shows that this condition is not fulfilled in general.
\begin{example}
Define $X(i,j)=(-1)^j i$ for $i\ge 1$ and $j\le 2$ and for the rest $X(i,j)=0$. Then the convergence in the rectangular sense of $\sum_{(i,j)\in\N_0^2} X(i,j)$ is clear, since 
$$\sum_{i,j=1}^{n} X(i,j)=0\quad \mbox{for all} \ n\ge 2.$$
But the series (A) $\sum_{i,j=0}^\infty \PP(\abs{X(i,j)}>c)$ diverges for all $c>0$. Notice further that $\eqref{redu}$ does not hold in this example.
\end{example}

In addition Klesov \cite{Klesov} shows that the condition \eqref{redu} can be dropped, if $(X_\nbf)_{\nbf\in\N_0^d}$ is symmetric or the random variables are positive. We state it in the following corollary. For our purposes, mainly the symmetric case is relevant.

\begin{corollar}[Klesov \cite{Klesov}, Corollaries 3 and 4]
Let $(X_\nbf)_{\nbf\in\N_0^d}$ be a real-valued random field of independent random variables. Then almost sure convergence of \eqref{new11} in the rectangular sense is equivalent to the convergence of $(A)$ and $(C)$ for some $c>0$, and hence for any $c>0$, if the random variables $(X_\nbf)_{\nbf\in\N_0^d}$ are symmetric, and equivalent to the convergence of (A) and (B) for some $c>0$, and hence for any $c>0$, if the random variables $(X_\nbf)_{\nbf\in\N_0^d}$ are positive.\label{thmklesov}
\end{corollar}

A random field $(Z_\tbf)_{\tbf\in\Z^d}$ is called {\it deterministic}, if there is a constant $K$ such that $\PP(Z_\tbf=K)=1$ for all $\tbf\in\Z^d$. Consider a nondeterministic real-valued i.i.d. stochastic process $(X_n)_{n\in\N}$ and coefficients $(\psi_n)_{n\in\N}\subset \R$. If the series $\sum_{n\in\N} \psi_n X_n$
converges almost surely absolutely, then by application of Theorem 5.1.4 of \cite{Chow} it can be concluded that $\sum_{n\in\N}\psi_n^2<\infty$. The same is true for multiple series, which converge almost surely in the rectangular sense.

\begin{theorem}
Suppose $(X_\nbf)_{\nbf\in\N_0^d}$ is a nondeterministic real-valued i.i.d. random field and $(\psi_\nbf)_{\nbf\in\N_0^d}$ are real coefficients. Furthermore suppose the random variables $(X_\nbf)_{\nbf\in\N_0^d}$ are symmetric or positive. If the multiple series $\sum_{\nbf\in\N_0^d} \psi_\nbf X_\nbf$ converges almost surely in the rectangular sense, then $\sum_{\nbf\in\N_0^d}\psi_\nbf^2<\infty$.

\label{lem3}
\end{theorem}

\begin{proof}
Suppose $\sum_{\nbf\in\N_0^d} \psi_\nbf X_\nbf$ converges almost surely in the rectangular sense. Then Corollary \ref{thmklesov} and the remark preceding it imply that
$$\PP\left(\abs{\psi_{\nbf_k}X_{\nbf_k}}>\epsilon\right)\to 0 \quad \forall \epsilon >0, \quad (k\to\infty), $$
for all sequences $(\nbf_k)_{k\in\N}$ with $\max(n_{1k},\ldots,n_{dk})\to\infty$. By assumption $(X_\nbf)_{\nbf\in\N_0^d}$ is an i.i.d. and nondeterministic random field. This implies that some $\epsilon>0$ exists such that
$$\PP(\abs{X_\nbf}>\epsilon)=\PP(\abs{X_{\textbf{0}}}>\epsilon)>0,$$
and in particular $X_{\nbf_k}$ does not converge in probability to zero, if $\max(n_{1k},\ldots,n_{dk})\to\infty$ for $k\to\infty$. Hence we can conclude $\psi_{\nbf_k}\to0$ for all sequences $(\nbf_k)_{k\in\N}$ with $\max(n_{1k},\ldots,n_{dk})\to\infty$.
Defining $Y_\nbf:=X_\nbf \indi_{\{\abs{\psi_\nbf X_\nbf}<1\}}$ we can conclude by Theorem \ref{theoremklesov2}
\begin{equation}
\sum_{\nbf\in\N_0^d} \psi_\nbf^2 \var(Y_\nbf)<\infty\label{jetze}
\end{equation}
(if $\psi_\nbf=0$, then $\ew Y_\nbf$ or $\var (Y_\nbf)$ need not to be defined, in which case we interpret $\var( Y_\nbf)$ as being infinity and $\psi_\nbf\var( Y_\nbf)$ to be equal to zero). Next, we claim that
\begin{equation}
\liminf_{N\to\infty}\min_{\abs{\nbf}\ge N}\var(Y_\nbf) >0.\label{newish}
\end{equation}
If this were not true, there must be a subsequence $(\nbf_{k})_{k\in\N}\subset \N_0^d$ such that $\max(n_{1k},\ldots,n_{dk})\to\infty$ and $\var(Y_{\nbf_k})\to 0$ as $k\to\infty$.
Thus, $Y_{\nbf_k}-\ew(Y_{\nbf_k})\to 0 \ (k\to\infty)$ in $L^2(\PP)$, and hence $X_{\boldsymbol 0} \indi_{\{\abs{\psi_{\nbf_k} X_{\boldsymbol 0}}<1\}}-\ew X_{\boldsymbol 0} \indi_{\{\abs{\psi_{\nbf_k} X_{\boldsymbol 0}}<1\}}$, which is equal in distribution to $Y_{\nbf_k}-\ew(Y_{\nbf_k})$, converges in probability to zero as $k\to\infty$. But this implies that $X_0$ is deterministic and we have a contradiction. Hence \eqref{newish} is satisfied and together with \eqref{jetze} we obtain $\sum_{\kbf\in\N_0^d} \psi_\kbf^2<\infty$.
\end{proof}\\

With the aid of the preceding theorem, we are able to prove necessary conditions for the existence of causal solutions. 
\begin{theorem}
Assume that  $(Z_\tbf)_{\tbf\in\Z^d}$ is i.i.d. nondeterministic noise and $R,S\subset\N_0^d\backslash\{\boldsymbol 0\}$ and that the ARMA equation \eqref{eq1} admits a causal solution $(Y_\tbf)_{\tbf\in\Z^d}$. Then the following two conditions are satisfied:\\
(i)
\begin{equation*}
\frac{\Theta(\zbf)}{\Phi(\zbf)}\in H^2,\quad \zbf\in\D^d.
\end{equation*}
(ii) Let $(Y'_\tbf,Z'_\tbf)_{\tbf\in\Z^d}$ be an independent copy of $(Y_\tbf,Z_\tbf)_{\tbf\in\Z^d}$ and define the symmetrizations 
\begin{equation}
(\tilde Y_\tbf,\tilde Z_\tbf):=\left\{\begin{aligned}(Y_\tbf,Z_\tbf)&,\quad &\PP_{Y_\textbf{0}}\ \text{and}\ \PP_{Z_{\textbf{0}}} \ \text{symmetric},\\
(Y_\tbf-Y'_\tbf,Z_\tbf-Z'_\tbf)&,\quad & \ \text{otherwise}.
\end{aligned}\right.
\end{equation}
Let further
\begin{equation}
\frac{\Theta(\zbf)}{\Phi(\zbf)}=\sum_{\kbf\in\N_0^d}\alpha_\kbf \zbf^{\kbf},\quad \zbf\in\D^d,\label{alex1}
\end{equation}
denote the power series expansion of $\Theta (\zbf)/ \Phi(\zbf)$, then $(\tilde Y_\tbf)_{\tbf\in\Z^d}$ is a solution of the symmetrized ARMA equation 
\begin{equation*}
\Phi(\Bbf)\tilde Y_\tbf=\Theta(\Bbf)\tilde Z_\tbf,\quad \tbf\in\Z^d,
\end{equation*}
and given by
$$\tilde Y_\tbf=\sum_{\kbf\in\N_0^d}\alpha_\kbf \tilde Z_{\tbf-\kbf},\quad \tbf\in\Z^d,$$
where the convergence of the right-hand side is almost surely rectangular. In particular, if $\PP_{Z_{\textbf{0}}}$ is symmetric, then there is at most one symmetric causal solution.\label{theorem2}
\end{theorem}

\begin{proof}
Assume $(Y_\tbf)_{\tbf\in\Z^d}$ is a causal solution of the ARMA equation \eqref{eq1}. 
Then the symmetrizations $(\tilde Y_\tbf)_{\tbf\in\Z^d}$ and $(\tilde Z_\tbf)_{\tbf\in\Z^d}$ fulfill the equation
$$\Phi(\Bbf)\tilde Y_\tbf=\Theta(\Bbf)\tilde Z_\tbf,\quad \tbf\in\Z^d,$$
and obviously $(\tilde Y_\tbf)_{\tbf\in\Z^d}$ is a causal solution of this equation. 
Reorganizing this ARMA equation, we get
$$\tilde Y_\tbf= \sum_{\textbf{n}\in R}\phi_\nbf \tilde Y_{\tbf-\nbf}+\tilde Z_\tbf+ \sum_{\textbf{n}\in S}\theta_\kbf \tilde Z_{\tbf-\kbf}, \quad \tbf\in\Z^d.$$
Now we replace each $\tilde Y_{\tbf-\nbf}$ on the right side by the ARMA equation. We do so as long as none of the random variables $\tilde Y_{\tbf-\sbf}, \sbf\le(N_1,\ldots,N_d)\in\N^d$ remains on the right-hand side of this equation.  
Defining for $\Nbf=(N_1,\ldots,N_d)\in\N_0^d$ the index sets
$$ I_\Nbf=\{\kbf\in\N^d_0: 0\le k_i\le N_i, i=1,\ldots,d\}, \quad B_\Nbf=\{\kbf\in\N^d_0: \kbf=\mbf+\nbf, \ \mbf\in I_{\boldsymbol N},\ \nbf\in R\cup S \}\backslash I_{\boldsymbol N},  $$
we get for $\Nbf=(N_1,\ldots,N_d)\in\N_0^d$ an equation like the following
\begin{align}
\tilde Y_\tbf&=\sum_{\nbf\in I_\Nbf}\alpha_{\nbf,\Nbf} \tilde Z_{\tbf-\nbf}+\sum_{\nbf\in B_{\Nbf}}\beta_{\nbf,\Nbf}\tilde Z_{\tbf-\nbf}+\sum_{\nbf\in B_{\Nbf}}\gamma_{\nbf,\Nbf}\tilde Y_{\tbf-\nbf}\notag\\
&=: A_{\tbf,\Nbf}+B_{\tbf,\Nbf}+C_{\tbf,\Nbf},\quad \Nbf\in \N^d_0,\label{eq3}
\end{align}
where $(\alpha_{\nbf,\Nbf})_{\nbf\in I_\Nbf}$, $(\beta_{\nbf,\Nbf})_{\nbf\in B_{\Nbf}^S}$,$(\gamma_{\nbf,\Nbf})_{\nbf\in B_{\Nbf}^R}$ are some complex coefficients and $A_{\tbf,\Nbf}$, $B_{\tbf,\Nbf}$ and $C_{\tbf,\Nbf}$ denote the first, second and third sum, respectively. Using the causality of $(\tilde Y_\tbf)_{\tbf\in\Z^d}$, notice that
$$A_{\tbf,\Nbf} \ind B_{\tbf,\Nbf}, C_{\tbf,\Nbf}\qquad \forall \Nbf\in\N^d_0,\quad \tbf\in\Z^d,$$
and furthermore we observe that
\begin{equation}
\alpha_{\nbf,\Nbf} =\alpha_{\nbf,\Nbf'}\quad   \forall \Nbf'>\Nbf,\quad \nbf\in I_\Nbf.\label{eq133}
\end{equation}
Because of equation \eqref{eq133} we write from now on rather $\alpha_\nbf$ than $\alpha_{\nbf,\Nbf}$. In the following we want to show that $|A_{\tbf,\Nbf}|$ is not converging in probability to infinity as $\min(N_1,\ldots,N_d)\to\infty$ using a technique adapted from \cite{Vollenbroeker}. The sum $|A_{\tbf,\Nbf}|$ is not converging in probability to infinity, if we can find some constants $K,\epsilon>0$ such that for every sequence $(\Nbf_k)_{k\in\N}, \Nbf_k=(N_{1k},\ldots, N_{dk})\in\N^d,$
with $\min(N_{1k},\ldots, N_{dk})\to\infty$ as $k\to\infty$ the following holds
\begin{equation}
\PP(\abs{A_{\tbf,\Nbf_k}}<K)\ge \epsilon, \quad \forall k\in\N.\label{eq4}
\end{equation}
By stationarity of $(\tilde Y_\tbf)_{\tbf\in\Z^d}$ there are some constants $K>0$ and $\epsilon\in(0,\frac{1}{4})$ such that
\begin{equation*}
\PP(\abs{\tilde Y_{\tbf}}<K)\ge 1-\epsilon\quad \forall\tbf\in\Z^d.
\end{equation*}
Using equation \eqref{eq3} this implies for $\epsilon \in(0,\frac{1}{4})$
\begin{equation*}
\PP(\abs{A_{\tbf,\Nbf_k}+B_{\tbf,\Nbf_k}+C_{\tbf,\Nbf_k}}<K)\ge 1-\epsilon, \qquad \forall k\in\N.
\end{equation*}
Using the inequalities $\abs{\Re(z)},\abs{\Im(z)}\le \abs{z}, z\in\C$, it follows
\begin{align}
\PP(\abs{\Re(A_{\tbf,\Nbf_k}+B_{\tbf,\Nbf_k}+C_{\tbf,\Nbf_k})}<K)&\ge 1-\epsilon,\quad \text{and}\notag\\
\PP(\abs{\Im(A_{\tbf,\Nbf_k}+B_{\tbf,\Nbf_k}+C_{\tbf,\Nbf_k})}<K)&\ge 1-\epsilon,  \qquad \forall k\in\N .\label{eq5a}
\end{align}
We will now show that the following holds for $c\in(\frac{1}{2},1-2\epsilon)$
\begin{align}
\PP(\abs{\Re(A_{\tbf,\Nbf_k})}<K)\ge c, \quad \PP(\abs{\Im(A_{\tbf,\Nbf_k})}<K)\ge c,  \qquad \forall k\in\N .\label{eq5}
\end{align}
Suppose \eqref{eq5} is not true. Then we can find $k_1, k_2\in\N$ such that one of the following inequalities is true
\begin{align*}
\PP(\abs{\Re(A_{\tbf,\Nbf_{k_1}})}\ge K)\ge 1-c, \quad \PP(\abs{\Im(A_{\tbf,\Nbf_{k_2}})}\ge K)\ge 1-c.
\end{align*}
So by the symmetry of $A_{\tbf,\Nbf}$ we have 
\begin{align}
\PP(\Re(A_{\tbf,\Nbf_{k_1}})\ge K)&\ge \frac{1-c}{2},\quad &\PP(\Re(A_{\tbf,\Nbf_{k_1}})\le -K)\ge \frac{1-c}{2}, \quad \text{or}\ \tag{3.12a}\label{casea} \\
\PP(\Im(A_{\tbf,\Nbf_{k_2}})\ge K)&\ge \frac{1-c}{2},\quad &\PP(\Im(A_{\tbf,\Nbf_{k_2}})\le -K)\ge \frac{1-c}{2},\quad\quad \ \tag{3.12b}\label{caseb}
\end{align}  \stepcounter{equation}
and obviously
\begin{equation*}
\PP(\Re(B_{\tbf,\Nbf_{k_1}}+C_{\tbf,\Nbf_{k_1}})\le 0)\ge\frac{1}{2} \quad \text{or} \quad \PP(\Re(B_{\tbf,\Nbf_{k_1}}+C_{\tbf,\Nbf_{k_1}})\ge 0)\ge\frac{1}{2}.
\end{equation*}
Suppose without losing the generaltity $\PP(\Re(B_{\tbf,\Nbf_{k_1}}+C_{\tbf,\Nbf_{k_1}})\le 0)\ge\frac{1}{2}$. By symmetry of $\Re(\tilde Y_{\tbf})=\Re(A_{\tbf,\Nbf_{k_1}}+B_{\tbf,\Nbf_{k_1}}+C_{\tbf,\Nbf_{k_1}})$ and independence of $A_{\tbf,\Nbf_{k_1}}$ and $B_{\tbf,\Nbf_{k_1}}+C_{\tbf,\Nbf_{k_1}}$ it follows in case \eqref{casea}
\begin{align}
\PP(\abs{\Re(A_{\tbf,\Nbf_{k_1}})+\Re(B_{\tbf,\Nbf_{k_1}}+C_{\tbf,\Nbf_{k_1}})}\ge K)&=2\PP(\Re(A_{\tbf,\Nbf_{k_1}})+\Re(B_{\tbf,\Nbf_{k_1}}+C_{\tbf,\Nbf_{k_1}})\le -K)\notag\\
&\ge 2\PP(\Re(A_{\tbf,\Nbf_{k_1}})\le -K, \Re(B_{\tbf,\Nbf_{k_1}}+C_{\tbf,\Nbf_{k_1}})\le 0)\notag\\
&=2\PP(\Re(A_{\tbf,\Nbf_{k_1}})\le -K)\ \PP(\Re(B_{\tbf,\Nbf_{k_1}}+C_{\tbf,\Nbf_{k_1}})\le 0)\notag\\
&\ge \frac{1-c}{2}>\epsilon.\label{eq11}
\end{align}
Similarly, in case \eqref{caseb} we obtain
\begin{eqnarray}
\PP(\abs{\Im(A_{\tbf,\Nbf_{k_2}})+\Im(B_{\tbf,\Nbf_{k_2}}+C_{\tbf,\Nbf_{k_2}})}\ge K)\ge  \frac{1-c}{2}>\epsilon.\label{eq12}
\end{eqnarray}
The equations \eqref{eq11} and \eqref{eq12} provide a contradiction to equation \eqref{eq5a}. Hence equation \eqref{eq5} is fulfilled and equation \eqref{eq4} follows easily (with possibly different constants $\epsilon, K$). In particular we showed that $\abs{\sum_{\nbf\in I_{\Nbf_k}}\alpha_\nbf \tilde Z_{\tbf-\nbf}}$ does not converge in probability to infinity as $k\to\infty$. Thus, by Theorem 3.17 of Kallenberg \cite{Kallenberg} we can conclude that $\sum_{\nbf\in I_{\Nbf_k}}\alpha_\nbf \tilde Z_{\tbf-\nbf}$ converges almost surely. Furthermore we notice by equation \eqref{eq3} that $B_{\tbf,\Nbf_{k}}+C_{\tbf,\Nbf_{k}}$ converges also almost surely as $k\to\infty$, and it is measurable with respect to $\sigma(\tilde  Z_\sbf, \sbf\le \tbf-\Nbf_l)$, $\forall l \in\N$. 
Hence we can deduce that the limit of $B_{\tbf,\Nbf_k}+C_{\tbf,\Nbf_k}$ is measurable with respect to the tail $\sigma$-field
$$	\bigcap_{k\in\N} \sigma(\tilde Z_\sbf, \sbf\le \tbf-\Nbf_k).$$
By Kolmogorov's zero-one law this  $\sigma$-field is $\PP$-trivial. Hence the limit of $B_{\tbf,\Nbf_k}+C_{\tbf,\Nbf_k}$ is almost surely constant, which we denote by $u$ for the moment. Altogether we have
$$\tilde Y_\tbf= u+\lim_{k\to\infty }\sum_{\nbf\in I_{\Nbf_k}}\alpha_\nbf\tilde Z_{\tbf-\nbf}\quad \text{a.s,}$$
for every sequence $(\Nbf_k)_{k\in\N}, \Nbf_k=(N_{1k},\ldots, N_{dk})\in\N^d,$ with $\min(N_{1k},\ldots, N_{dk})\to\infty$ as $k\to\infty$. By symmetry of $(\tilde Y_\tbf)_{\tbf\in\Z^d}$ and $(\tilde Z_\tbf)_{\tbf\in\Z^d}$, we must have $u=0$ almost surely. Further, $\sum_{\nbf\in \N_0^d}\alpha_\nbf \tilde Z_{\tbf-\nbf}$ converges almost surely in the rectangular sense. Applying Theorem \ref{lem3} it follows that 
$$\sum_{\nbf\in \N^d_0}\abs{\alpha_\nbf}^2<\infty.$$

The proof is finished if we can show that $\Theta(\zbf)/\Phi(\zbf)\in H^2$ and the power series expansion of this function is given by
\begin{equation}
\frac{\Theta(\zbf)}{\Phi(\zbf)}=\sum_{\kbf\in\N_0^d}\alpha_\kbf \zbf^{\kbf},\quad \zbf\in\D^d.\label{last}
\end{equation}
To do so, we apply the operator $\Phi(\Bbf)$ to both sides of equation \eqref{eq3} and get by \eqref{eq1}
\begin{align*}
\Phi(\Bbf)\tilde Y_\tbf&=\sum_{\nbf\in I_\Nbf}\alpha_\nbf \Phi(\Bbf)\tilde Z_{\tbf-\nbf}+\sum_{\nbf\in B_{\Nbf}}\beta_{\nbf,\Nbf} \Phi(\Bbf)\tilde Z_{\tbf-\nbf}+\sum_{\nbf\in B_{\Nbf}}\gamma_{\nbf,\Nbf} \Theta(\Bbf) \tilde Z_{\tbf-\nbf}\notag\\
&=\Theta(\Bbf) \tilde Z_\tbf, \quad\Nbf=(N_1,\ldots, N_d)\in\N_0^d.
\end{align*}
Because of the independence of $(\tilde Z_\tbf)_{\tbf\in\Z^d}$ and the assumption that it is nondeterministic we can conclude that the coefficients of both sides of this equation are equal. Thus for each $\Nbf=(N_1,\ldots,N_d)$ with $I_{\boldsymbol N}\supset S$ we have $\alpha_{\textbf{0}}=1$ and
\begin{equation*}
\alpha_\kbf-\sum_{\nbf\in R,\\ \nbf \le \kbf} \phi_{\nbf}\alpha_{\kbf-\nbf}=\left\{
\begin{aligned}
\theta_\kbf\quad, &\ \kbf\in S,\\
0\quad, &\ \kbf\in I_\Nbf\text{ \textbackslash} (S\cup \{\textbf{0}\}).
\end{aligned}
\right.
\end{equation*}
Hence we observe for all $\tbf\in\td$
\begin{align}
\left(1-\sum_{\nbf\in R}\phi_\nbf e^{-i\nbf\tbf}\right)\left(\sum_{\kbf\in\N_0^d} \alpha_\kbf e^{-i\kbf\tbf} \right)
=&\sum_{\kbf\in\N_0^d} \left(\alpha_\kbf -\sum_{\nbf\in R, \nbf\le \kbf}\phi_\nbf \alpha_{\kbf-\nbf} \right)  e^{-i\kbf\tbf}\notag\\
=&1+ \sum_{\nbf\in S} \theta_\nbf  e^{-i\nbf \tbf}=\Theta(e^{-i\tbf}).\label{gll}
\end{align}
The zero set of $\Phi(e^{-i\cdot})$ is a null set of the $d$-dimensional Lebesgue measure. Because of this and equation \eqref{gll}, we conclude $\Theta(e^{-i\cdot})/\Phi(e^{-i\cdot})\in H^2 \subset L^2(\td)$ and that is equation \eqref{last}. 
 \end{proof}\bigbreak
In the case that $\PP_{Z_{\textbf{0}}}$ is symmetric, we can now give the following necessary and sufficient condition for the existence of symmetric causal solutions:
\begin{corollar}
Assume $(Z_\tbf)_{\tbf\in\Z^d}$ is i.i.d. nondeterministic symmetric noise and $R,S\subset \N_0^d\backslash\{\boldsymbol 0\}$. Then the ARMA equation \eqref{eq1} admits a causal solution $(Y_\tbf)_{\tbf\in\Z^d}$ if and only if
\begin{equation*}
\frac{\Theta(\zbf)}{\Phi(\zbf)}\in H^2,\quad \zbf\in\D^d,
\end{equation*}
and $\sum_{\kbf\in\N_0^d} \alpha_\kbf Z_{\tbf-\kbf}$ converges almost surely in the rectangular sense, where $(\alpha_\kbf)_{\kbf\in\N_0^d}$ is given by \eqref{alex1}. In that case, $Y_\tbf:=\sum_{\kbf\in\N_0^d} \alpha_\kbf Z_{\tbf-\kbf}$ defines the unique symmetric causal solution.
\end{corollar}
\begin{proof}
Necessity has been shown in Theorem \ref{theorem2}, and sufficiency follows as in the proof of Theorem \ref{thm1}.
\end{proof}\bigbreak
In the following theorem we give sufficient conditions for the existence of a causal solution. Notice that under condition $(i)$ the convergence is even almost surely absolutely.

\begin{theorem}
A causal solution of the ARMA equation exists, if one of the following conditions is fulfilled:
\begin{eqnarray*}
&(i)&\quad\ew \log_+^d \abs{Z_{\textbf{0}}}<\infty  \quad \mbox{and}\quad \Phi(\zbf)\neq 0 \quad \forall \zbf\in\overline\D^d.\\
&(ii)&\quad\ew \abs{Z_{\textbf{0}}}^2<\infty, \ \ew Z_{\textbf{0}}=0 \quad \mbox{and}\quad \frac{\Theta(\zbf)}{\Phi(\zbf)}\in H^2.
\end{eqnarray*}
In both cases the quotient $\Theta(\zbf)/\Phi(\zbf)$ admits a power series expansion, given by
\begin{equation*}
\frac{\Theta(\zbf)}{\Phi(\zbf)}=\sum_{\kbf\in\N^d_0} \psi_\kbf \zbf^\kbf, \quad \zbf=(z_1,\ldots, z_d) \in \D^d,
\end{equation*}
and a causal solution is given by
\begin{equation}
Y_\tbf=\sum_{\kbf\in\N^d_0}\psi_\kbf Z_{\tbf-\kbf}, \quad\tbf\in\Z^d.\label{solli}
\end{equation}\label{propwrong}
\end{theorem}
\begin{proof}
In case $(i)$ by the condition that $\Phi(\zbf)$ has no zero on $\overline\D^d$ the existence of a multidimensional power series expansion is assured. The remaining proof is almost the same as the proof of Proposition \ref{prop111}. The assumptions in case $(ii)$ assure $L^2(\PP)$-convergence of (\ref{solli}). Furthermore it is easy to see that \eqref{solli} solves the ARMA equations in both cases.
\end{proof}\bigbreak

Having derived necessary conditions and sufficient conditions, we want to discuss the crucial condition that the quotient of the ARMA polynomials lies in $H^2$. In the time series model ($d=1$), if $\Phi(z)$ and $\Theta(z)$ have no common zeros, a necessary and sufficient condition for the existence of strictly stationary solution is $\Phi(z)\neq 0$ for $\abs{z}=1$ and $\ew \log_+ Z_0<\infty$ (if $\Phi$ is not constant), see Brockwell and Lindner \cite{Brockwelllindner}.  We will see that in contrast to the time series model for $d>1$ it is not necessary that the analog condition $\Phi(e^{-i\tbf})\neq 0$ for $\tbf\in \T^d$ holds, cf. the upcoming Example \ref{contrast}.\bigbreak
A polynomial in two or more variables can in general not be factored as in one variable. If a polynomial $p(x_1,\ldots,x_n)$ admits a factorization $p=qr$, where $q$ and $r$ are nonconstant polynomials of $n$ or less variables, then it is called {\it reducible}, otherwise {\it irreducible}. Every polynomial of several variables admits a factorization into irreducible factors, which is essentially, except for multiplication with constants, unique, cf. Bôcher \cite{Bocher}, Chapter 16. If this factorization consists of only one nonconstant irreducible factor, then the polynomial is irreducible.\\
The following result will be useful to exclude zeros of $\Phi$ on the closed unit disc if $d=2$ and $\Theta\equiv 1$, cf. Corollary \ref{2dim}.
\begin{theorem}
Suppose $\Phi:\C^d\to\C$ is an irreducible polynomial in $d\ge 2$ variables and further, if arbitrary $d-1$ variables are fixed, the polynomial in the remaining variable is not identically zero. If $\Phi$ has a root $t^{(0)}=(t_1^{(0)},\ldots,t_d^{(0)})\in(\partial\mathbb{D}^d)\backslash \td$, i.e. $\Phi(t^{(0)})=0$, then $\Phi$ has also roots inside the open unit polydisc $\mathbb{D}^d$.\label{zero}
\end{theorem}
\begin{proof}
Suppose $t^{(0)}=(t_1^{(0)},\ldots,t_d^{(0)})\in(\partial\mathbb{D}^d)\backslash \td$ is a root of $\Phi$. Then for at least one $i\in\{1,\ldots,d\}$ we have $t_i^{(0)}\in\D$, and for at least one $i\in\{1,\ldots,d\}$ we have $t_i^{(0)}\in\T$. Define
\begin{eqnarray*}
I:=\left\{i\in\{1,\ldots,d\}: t_i^{(0)}\in\T \right\},\quad N:=\left\{i\in\{1,\ldots,d\}: t_i^{(0)}\in\D \right\},\quad I\cup N=\{1,\ldots,d\}.
\end{eqnarray*}
Without loss of generalization, we assume $I=\{1,\ldots,j\}$ and $N=\{j+1,\ldots,d\}$. We fix all variables except $t_{j}\in I$ and $t_d\in N$ and consider the two variable polynomial
$$p(t_{j},t_d):=\Phi(t_1^{(0)},\ldots,t_{j-1}^{(0)},t_j,t_{j+1}^{(0)},\ldots,t_{d-1}^{(0)},t_d)=\sum_{k=0}^n a_k(t_j) t_d^k,\quad n\in\N,$$
where the coefficients $a_k(t_j)$ are themselves polynomials in one variable $t_j$. We have $p(t^{(0)}_{j},t^{(0)}_d)=0$. Suppose $a_n(t^{(0)}_{j})\neq 0$. Then, by Theorem 3.9.1 of \cite{Tyrtyshnikov}, the polynomial roots $x_i(t),i=1,\ldots,n$ of the equation $p(t,x_i(t))=0$ can be chosen to be continuous in $t$ in a neighborhood of $t_j^{(0)}$. Thus, there exists $t^{(1)}_j,t^{(1)}_d\in\D$ such that $p(t^{(1)}_j,t^{(1)}_d)=0$. Now suppose $a_n(t^{(0)}_{j})= 0$. Then there exists $1\le k<n$ such that $a_k(t^{(0)}_{j})\neq 0$ or otherwise the polynomial $p(t^{(0)}_{j},\cdot)$ is identically zero, which is excluded by the assumptions of the theorem. In the first case we can use the same argument as before. Applying this argument inductively for all $i\in I$ yields the statement of the theorem.
\end{proof}\\

If we consider a polynomial $\Phi:\C^2\to\C$ in two variables, the assumption that $\Phi$ is irreducible implies the second condition in the preceding theorem: if $\Phi(z_1,z_2)=\sum_{k=0}^n a_k(z_1) z_2^k$ is identically zero, the first variable being fixed, then the coefficients $a_k$, which are polynomials themselves, have a common zero and hence the polynomial can be factorized. Thus, by Theorem \ref{zero} for $d=2$ a root in $(\partial\D^2)\backslash \T^2$ of an irreducible polynomial implies a root inside $\D^2$. However, in three variables this is not true, consider e.g.
$$\Phi(z_1,z_2,z_3)=(1-z_1)z_3+(1-z_2)z_3^2.$$
Fixing $(z_1, z_2)=(1,1)$ the polynomial $\Phi(1,1,z_3)$ is identical to zero, but $\Phi$ can not be factorized. Notice that for $\Phi^{-1}\in H^2$ it is necessary and sufficient that $\Phi(\zbf)\neq0$ on $\D^d$ and 
\begin{eqnarray}
\int_{\td}\left|\frac{1}{\Phi(e^{-i\tbf})}\right|^2 d\lambda^d(\tbf)<\infty\label{finit}.
\end{eqnarray}
where the integral does not depend on the values on $(\partial\mathbb{D}^d)\backslash \td$. With the help of Theorem \ref{zero}, we now establish the following necessary condition in the spatial autoregressive model for $d=2$:
\begin{corollar}
Suppose $(Z_\tbf)_{\tbf\in\Z^2}$ is i.i.d. nondeterministic and $R\subset\N_0^2\backslash\{\boldsymbol 0 \}$. A necessary condition for the existence of a causal solution in the autoregressive model \eqref{eq1}, where $\Theta(z_1,z_2)\equiv 1$, is given by
$$\Phi(z_1,z_2)\neq0 \quad\forall (z_1,z_2)\in\overline\D^2.$$\label{2dim}
\end{corollar}
\begin{proof}
By Theorem \ref{theorem2} we know that $\Phi^{-1}(z_1,z_2)\in H^2$ is a necessary and sufficient condition for the existence of a causal solution. This implies directly that $\Phi(z_1,z_2)$ can not possess any root on $\D^2$. Now assume $\Phi(z_1,z_2)=\prod_{i=1}^n \Phi_i(z_1,z_2)$ is a factorization of $\Phi$, where each $\Phi_i(z_1,z_2)$ is irreducible. If one factor $\Phi_i(z_1,z_2)=\Phi_i(z_1)$ only depends on one variable, then it follows
$$\Phi_i(z_1)\neq 0 \quad \forall z_1\in \C: \abs{z_1}\le 1,$$
since, if $\abs{z_1}<1$ is a zero, then $\Phi(z_1,z_2)$ will have a zero on $\D^2$. If $z_1\in\C$ with $\abs{z_1}=1$ is a root, then $\Phi^{-1}(z_1,z_2)$ can not be square integrable. If a factor $\Phi_i$ depends on two variables, then we can apply Theorem \ref{zero} and it follows $\Phi_i(z_1,z_2)\neq0$ on $(\partial\D^2)\backslash \T^2$. It remains to show that  $\Phi(e^{-i\cdot},e^{-i\cdot})\neq0$ on $\T^2$.
Suppose $\boldsymbol w\in\T^2$ is a zero of $\Phi(e^{-i\cdot},e^{-i\cdot})$. Then by the mean value theorem, for an arbitrary norm $|| \cdot ||$ on $\R^2$ and some $C>0$
$$\abs{\Phi(e^{-i(\boldsymbol w +\hbf)})}=\abs{\Phi(e^{-i(\boldsymbol w +\hbf)})-\Phi(e^{-i\boldsymbol w})}\le C ||\hbf||,\quad \hbf\in\T^2.$$
But this implies
\begin{equation}
\int_{\mathbb{T}^2} \frac{d\hbf}{|\Phi(e^{-i(\boldsymbol w +\hbf)})|^2}\ge C^{-1}\int_{\mathbb{T}^2} \frac{d\hbf}{||\hbf||^2}=\infty,\label{added}
\end{equation}
where the latter integral is infinite by simple calculus. This contradicts $\Phi^{-1}\in H^2$. Altogether we showed that no zero on $\overline\D^2$ can exist.
\end{proof}\bigbreak
We go on discussing the relation between zeros on $\T^d$ and the finiteness of \eqref{finit}. The following example from \cite{Rosenblatt} shows that for $d\ge3$ it is possible to have roots on $\td$ and still $\Phi^{-1}(\zbf)\in H^2$ holds.

\begin{example}
Consider for $d=5$ the function 
$$\Phi(\zbf)=1-\frac{1}{5}\sum_{i=1}^5 z_i,\quad \zbf=(z_1,\ldots,z_5)\in\C^5.$$
We will show $\Phi^{-1}(e^{-i\tbf})\in L^2(\T^5)$, and this implies $\Phi^{-1}(\zbf)\in H^2$ by noticing that the only root in $\D^5$ is $\indi=(1,\ldots,1)$. 
Utilizing the Taylor expansion
\begin{eqnarray*}
e^{-ih_j}=1-ih_j-h_j^2+\mathcal{O}(h_j^3),\quad h_j \to 0,\quad j=1,\ldots,5,
\end{eqnarray*}
we estimate 
\begin{eqnarray*}
\abs{\Phi(e^{-i\hbf})}^2=\left|\frac{1}{5}\sum_{j=1}^5 (e^{-ih_j}-1)\right|^2&=&\left|\frac{1}{5}\sum_{j=1}^5\left(ih_j+h_j^2+\mathcal{O}(h^3_j)\right)\right|^2\\
&=&\left|\frac{1}{5}\sum_{j=1}^5 (ih_j+h_j^2)\right|^2\left|1+\frac{\sum_{j=1}^5\mathcal{O}(h^3_j)}{\frac{1}{5}\sum_{j=1}^5 (ih_j+h_j^2)}\right|^2.
\end{eqnarray*}
Observe that
$$\frac{\sum_{j=1}^5\mathcal{O}(h^3_j)}{\frac{1}{5}\sum_{j=1}^5 (ih_j+h_j^2)}\to 0 \quad \text{as} \ \max(|h_1|,\ldots,|h_5|)\to 0.$$
Hence, there are $\epsilon>0$ and $C_1,C_2>0$ such that
\begin{eqnarray*}
\int_{\mathbb{T}^5} \frac{d\hbf}{\abs{\Phi(e^{-i\hbf})}^2}\le C_1+\int_{||\hbf||<\epsilon} \frac{d\hbf}{\abs{\Phi(e^{-i\hbf})}^2}&\le& C_1+C_2\int_{||\hbf||<\epsilon} \frac{d\hbf}{\left|\frac{1}{5}\sum_{j=1}^5 (ih_j+h^2_j)\right|^2}\\
&\le&C_1+25C_2\int_{||\hbf||<\epsilon} \frac{d\hbf}{||\hbf||^4}<\infty,
\end{eqnarray*}\label{contrast}
where $\norm{\cdot}$ denotes the Euclidean norm.
\end{example}

Rosenblatt \cite{Rosenblatt2}, p. 228, states that the reciprocal of the similar polynomial $\Phi(z_1,z_2,z_3)=1-\frac{1}{3}(z_1+z_2+z_3)$ for $d=3$ is also in $H^2$.  \bigbreak While for $d=2$ for autoregressive models $\Phi(e^{-i\tbf})\neq0$ for all $\tbf\in\T^2$ is necessary, this is no longer the case for ARMA models, as the following example shows.
\begin{example}
Consider the two-dimensional ARMA model 
$$Y_{\tbf}-\frac{1}{2}B_1 Y_{\tbf}-\frac{1}{2} B_2 Y_{\tbf}=Z_{\tbf}-B_1Z_{\tbf}-B_2Z_{\tbf}+B_1B_2Z_{\tbf}, \quad \tbf=(t_1,t_2)\in\Z^2.$$
The corresponding moving average and autoregressive polynomials are given by
\begin{eqnarray*}
\Theta(z_1,z_2)&=&(1-z_1)(1-z_2),\\
\Phi(z_1,z_2)&=&1-\frac{1}{2} z_1 -\frac{1}{2} z_2,\quad z_1,z_2\in\C.
\end{eqnarray*}
Notice that both polynomials have a common zero $(z_1,z_2)=(1,1)$ on $\mathbb{T}^2$. We define $H^\infty$ as usually as the vector space of all holomorphic functions $f:\D^d\to\C$, which are bounded on $\D^d$. Then $\Theta(z_1,z_2)/ \Phi(z_1,z_2)\in H^\infty\subset H^2$, as can be seen by following estimation:
\begin{eqnarray*}
\left|\frac{\Theta(z_1,z_2)}{\Phi(z_1,z_2)}\right|^2&=&\left|\frac{2(1-z_1)(1-z_2)}{1-z_1+1-z_2}\right|^2
=4\frac{1}{\abs{\frac{1}{1-z_2}+\frac{1}{1-z_1}}^2}\le 4 \frac{1}{[\Re(\frac{1}{1-z_2}+\frac{1}{1-z_1})]^2}\le 4,\quad \forall z_1,z_2\in\D.
\end{eqnarray*}
\end{example}
In difference to the one-dimensional case the common root of the nominator and denominator can not be canceled out, because the polynomials do not factorize. In some sense, the zero of the nominator covers for the zero of the denominator, resulting in the square integrability.


\section{The spatial autoregressive model of first order}\setcounter{equation}{0}
In the foregoing section we were able to specify necessary conditions and sufficient conditions for the existence of causal solutions, in terms of the zero set of the ARMA polynomials. However, we could not give necessary moment conditions on the noise $(Z_\tbf)_{\tbf\in\Z^d}$. It turns out that in difference to the one-dimensional case, where the asymptotics of the coefficients of the Laurent expansion 
$$\frac {\Theta(z)}{\Phi(z)}=\sum_{k\in\Z} \psi_k z^k,$$
can be easily completely determined in dependence of the zeros of $\Phi(z)$, for $d>1$ it is difficult to determine the asymptotics of the corresponding Laurent or power series expansion. To specify necessary moment conditions, lower bounds on the decay of the coefficients are needed.
However, even though in general it seems difficult to determine lower bounds or the exact asymptotics, for a specific model we are able to determine necessary moment conditions. In this section we want to establish a full characterization of necessary and sufficient conditions for the existence of causal solutions of the autoregressive model of first order with real coefficients in dimension two:\bigbreak
Consider the spatial autoregressive model defined by the equations
\begin{equation}
Y_{t_1,t_2}-\phi_1Y_{t_1-1,t_2}-\phi_2 Y_{t_1,t_2-1}-\phi_3 Y_{t_1-1,t_2-1}=Z_{t_1,t_2}, \quad (t_1,t_2)\in\Z^2,\label{model}
\end{equation}
where $\phi_1,\phi_2,\phi_3\in\R$, $(\phi_1,\phi_2,\phi_3)\neq(0,0,0)$ and $(Z_\tbf)_{\tbf\in\Z^2}$ is an i.i.d. complex-valued random field. We want to establish necessary and sufficient conditions for the existence of a causal solution. To do so, several auxiliary results are needed.
First, we want to determine the coefficients, which solve the to the model \eqref{model} corresponding partial difference equation
\begin{eqnarray}
\psi_{n,k}&=&\phi_1\psi_{n-1,k}+\phi_2\psi_{n,k-1}+\phi_3\psi_{n-1,k-1},\quad n,k\in\N_0,\quad (n,k)\neq (0,0),\label{pde}\\
\psi_{n,0}&=&\phi_1^n,\quad \psi_{0,k}=\phi_2^k \quad \text{for} \quad n,k\in\N_0,\label{pde2}
\end{eqnarray}
where \eqref{pde2} are the boundary conditions and convention $ \psi_{n,k}=0$ for $(n,k)\in\Z^2\backslash\N_0^2$ is used. It is associated with \eqref{model} by the equation
$$\left( \sum_{n,k=0}^\infty \psi_{n,k} z_1^n z_2^k\right)(1-\phi_1 z_1-\phi_2 z_2-\phi_3 z_1z_2)=1.$$
The solution of this partial difference equation is determined in \cite{Fray}:
\begin{lemma}
The unique solution of \eqref{pde} with boundary conditions \eqref{pde2} is given by
\begin{eqnarray}
\psi^{\phi_1,\phi_2,\phi_3}_{n,k}:=\psi_{n,k}&=&\sum_{j=0}^{n} \binom{k}{j}\binom{n+k-j}{k} \phi_1^{n-j} \phi_2^{k-j}\phi_3^{j}\label{formula1}\\
&=&\sum_{j=0}^n \binom{n}{j}\binom{k}{j} \phi_1^{n-j} \phi_2^{k-j} (\phi_1\phi_2+\phi_3)^{j}\label{formula2},\quad n,k\in\N_0.
\end{eqnarray}\label{lllast}
\end{lemma}
The formulas \eqref{formula1} and \eqref{formula2} are also valid if some of the coefficients $\phi_1,\phi_2,\phi_3$ are equal to zero.
The numbers $\psi^{\phi_1,\phi_2,\phi_3}_{n,k}$ are called {\it weighted Delannoy numbers}. They can be interpreted as the number of weighted paths from $(0,0)$ to $(n,k)$ in the two-dimensional lattice $\N_0^2$, when only moving with steps $(1,0)$, $(0,1)$ and $(1,1)$ is allowed and related weights $\phi_1$, $\phi_2$ and $\phi_3$ respectively. To establish later on necessary moment conditions on the noise $(Z_\tbf)_{\tbf\in\Z^2}$, the asymptotics of $\psi^{\phi_1,\phi_2,\phi_3}_{n,k}$ for $n,k\to\infty$ have to be known. Hetyei \cite{Hetyei} discovered that the weighted Delannoy numbers are related to Jacobi polynomials. For $n\in\N$ the $n$-th {\it Jacobi polynomial $P_n^{(\alpha,\beta)}(x)$ of type $(\alpha,\beta), \alpha,\beta>-1,$} is defined as
$$P_n^{(\alpha,\beta)}(x)=(-2)^{-n} (n!)^{-1} (1-x)^{-\alpha}(1+x)^{-\beta} \frac{d^n}{dx^n} \left((1-x)^{n+\alpha}(1+x)^{n+\beta}\right),\quad x\in(-1,1).$$
The Jacobi polynomial $P_n^{(\alpha,\beta)}(x)$ is indeed a polynomial of degree $n$ on $(-1,1)$ and can therefore be extended to $x\in\R$.
The following relationship is valid, see Theorem 2.8 of \cite{Hetyei}:

\begin{theorem} For $\beta\in\N_0$ and $\phi_3\neq 0$ we have
$$\psi^{\phi_1,\phi_2,\phi_3}_{k+\beta,k}=\phi_1^\beta (-\phi_3)^k  P^{(0,\beta)}_k\left(-2\frac{\phi_1\phi_2}{\phi_3}-1\right),\quad k\in\N_0.$$
\label{prop1}
\end{theorem}

Therefore the asymptotic behaviour of the coefficients depends on the asymptotics of the Jacobi polynomials, which have been studied extensively. Wong and Zhao [22, Theorem 5.1] established an asymptotic expansion for Jacobi polynomials with explicit error term. Accordingly, the asymptotic expansion of order $p\in\N$ can for $N:=n+\frac{\beta+1}{2}$ with $\beta\in\N_0$ be written as
\begin{eqnarray}
\left(\cos\frac{\theta}{2}\right)^{\beta}P^{(0,\beta)}_n\left(\cos\theta\right)= -J_0(N\theta)\sum_{k=0}^{p-1} \frac{c_k(\theta)}{N^k}-J_1(N\theta)\sum_{k=0}^{p-1} \frac{d_k(\theta)}{N^k}+\delta_p(N,\theta),\label{asympexp}
\end{eqnarray}
where $J_\mu$ is the Bessel function of order $\mu$, $c_k(\theta)$ and $d_k(\theta)$ are some coefficients and $\delta_p(N,\theta)$ is the error term. The asymptotic expansion \eqref{asympexp} holds uniformly in $\theta\in(0,\pi)$ and the coefficients can be calculated explicitly. The first coefficients are given by $c_0(\theta)=-\theta^{\frac{1}{2}} (\sin\theta)^{-\frac{1}{2}}$ and $d_0(\theta)=0$. Thus, for $p=1$ the asymptotic expansion equals
\begin{eqnarray}
P^{(0,\beta)}_n\left(\cos\theta\right)=\left(\cos\frac{\theta}{2}\right)^{-\beta}\left(\sqrt{ \frac{\theta}{\sin\theta}} J_0(N\theta)+\delta_1(N,\theta)\right),\label{3dim1}
\end{eqnarray}
where the error term can be estimated by 
\begin{eqnarray}
|\delta_1(N,\theta)|\le\frac{\Lambda}{N} \left(|J_0(N\theta)|+|J_1(N\theta)|\right),\quad \Lambda>0.\label{3dim}
\end{eqnarray}
Here, the constant $\Lambda$ is independent of $\theta$, $n$ and $\beta$.
Now we are prepared to establish the estimation from below of the asymptotics of the coefficients $(\psi_{n,k})_{(n,k)\in\N_0^2}$ to determine moment conditions. Notice that $\phi_1,\phi_2,\phi_3\in(-1,1)$ is necessary for $\Phi^{-1}(z_1,z_2)=(1-\phi_1z_1-\phi_2z_2-\phi_3 z_1 z_2)^{-1}\in H^2$, see Corollary \ref{2dim} and Basu and Reinsel \cite{Basu}, Proposition 1.
\begin{lemma}
Let $\phi_1,\phi_2,\phi_3\in(-1,1)$ and at least two coefficients not equal to zero. Then there is a constant $C>0$ and $x_0>1$ such that 
$$f(x):=\left|\left\{ (n,k)\in\N^2_0 :  \abs{\psi^{\phi_1,\phi_2,\phi_3}_{n,k}}^{-1}\le x \right\}\right|\ge C \log^2(x), \quad x>x_0>1.$$\label{Last}
\end{lemma}
\begin{proof}
We consider the three different cases $(i)$\  $\phi_1=0$ or $\phi_2=0$, $(ii)$\ $\phi_1\phi_2\neq0$  and  $\phi_3\phi_1^{-1}\phi_2^{-1}\ge-1$ and $(iii)$\ $\phi_1\phi_2\neq0$  and $\phi_3\phi_1^{-1}\phi_2^{-1}<-1$. Firstly, observe for $\phi_1\phi_2\neq0$ by \eqref{formula2}
$$\psi^{\phi_1,\phi_2,\phi_3}_{n,k}= \phi_1^n \phi_2^k\sum_{j=0}^n \binom{n}{j}\binom{k}{j}  (1+\frac{\phi_3}{\phi_1\phi_2})^{j}.$$
Thus, in the case $(ii)$  \ $\phi_1\phi_2\neq0$  and  $\phi_3\phi_1^{-1}\phi_2^{-1}\in[-1,\infty)$ we have $\abs{\psi^{\phi_1,\phi_2,\phi_3}_{n,k}}\ge \abs{\phi_1}^n \abs{\phi_2}^k$ for all $(n,k)\in\N_0^2$ and therefore for some $x_0>1$ and $C>0$
$$f(x)\ge\left|\left\{ (n,k)\in\N^2_0 : \abs{\phi_1}^{-n}\abs{\phi_2}^{-k}\le x \right\}\right|\ge C \log^2(x),\quad x\ge x_0,$$
where the inequality follows easily. Secondly, assume in the case $(i)$\  $\phi_1=0$ or $\phi_2=0$ without restricting the generality $\phi_1=0$ (the case $\phi_2=0$ follows then by symmetry). By equation \eqref{formula1}, the coefficients $(\psi_{n,k})_{(n,k)\in\N_0^2}$ are given by
\begin{eqnarray*}
\psi_{n,k}=\left\{
\begin{aligned}
\binom{k}{n} \phi_2^{k-n}\phi_3^n&, \quad &k\ge n,\\
0&,\quad &k<n.
\end{aligned}
\right.
\end{eqnarray*}
Therefore we estimate
\begin{eqnarray*}
f(x)&=&\abs{\{ (n,k)\in\N^2_0, k\ge n:  \binom{k}{n}^{-1} \abs{\phi_2}^{-(k-n)}\abs{\phi_3}^{-n}\le x \}}\\
&\ge&\abs{\{ (n,k)\in\N^2_0, k\ge n:   \min(\abs{\phi_2},\abs{\frac{\phi_3}{\phi_2}})^{-(k+n)}\le x \}}\\
&\ge&\frac{1}{2}\abs{\{ (n,k)\in\N^2_0:   \min(\abs{\phi_2},\abs{\frac{\phi_3}{\phi_2}})^{-(k+n)}\le x \}}\ge C \log^2(x),
\end{eqnarray*}
for all  $x>x_0>1$ and some $C>0$.

At last, consider the case $(iii)$\ $\phi_1\phi_2\neq0$  and  $\phi_3\phi_1^{-1}\phi_2^{-1}< -1$ and define $z:=-2\phi_1\phi_2^{-1}\phi_3^{-1}-1\in(-1,1)$ and $\theta\in(0,\pi)$ by the equation $\cos\theta=z$. Define further 
$$A_x^\beta:=\{k\in\N_0: \abs{\psi^{\phi_1,\phi_2,\phi_3}_{k+\beta,k}}^{-1}\le x\}.$$
Then the function $f$ satisfies $f(x)\ge\sum_{\beta=0}^\infty |A_x^\beta|$.

Notice that the equality $\psi^{\phi_1,\phi_2,\phi_3}_{n,k}=\phi_1^n \phi_2^k \psi_{n,k}^{1,1,\frac{\phi_3}{\phi_1\phi_2}}$ holds. Hence by Theorem \ref{prop1} we can express the set $A^\beta_x$ as
\begin{eqnarray*}
A_x^\beta&=&\left\{n\in\N_0:\abs{\phi_1}^{-(n+\beta)} \abs{\phi_2}^{-n} \abs{\psi^{1,1,\frac{\phi_3}{\phi_1\phi_2}}_{n+\beta,n}}^{-1}\le x\right\}\\
&=&\left\{n\in\N_0:\abs{\phi_1}^{-(n+\beta)} \abs{\phi_2}^{-n} \abs{\frac{\phi_3}{\phi_1\phi_2}}^{-n} \left|P^{(0,\beta)}_n\left(-\frac{2\phi_1\phi_2}{\phi_3}-1\right)\right|^{-1}\le x\right\}.
\end{eqnarray*}
In the following we shall estimate
\begin{eqnarray*}
\abs{\frac{\phi_3}{\phi_1\phi_2}}^{n} \left|P^{(0,\beta)}_n\left(-\frac{2\phi_1\phi_2}{\phi_3}-1\right)\right|
\end{eqnarray*}
from below using the asymptotic expansion \eqref{3dim1} with error term \eqref{3dim}. Hence we have
\begin{eqnarray}
\abs{P^{(0,\beta)}_n\left(\cos\theta\right)}&=&\abs{\cos \frac{\theta}{2}}^{-\beta}\left|\sqrt{ \frac{\theta}{\sin\theta}} J_0(N\theta)+\delta_1(N,\theta)\right|,\quad \beta\in\N_0.\label{d1}
\end{eqnarray}
Denote $3\delta=\min(\theta,\pi-\theta)$. Then if $(n+\frac{\beta+1}{2})\theta-\frac{\pi}{4}\in\cup_{k\in\Z}(\frac{\pi}{2}+k\pi-\delta,\frac{\pi}{2}+k\pi+\delta)$, then $(n+1+\frac{\beta+1}{2})\theta-\frac{\pi}{4}\not\in\cup_{k\in\Z}(\frac{\pi}{2}+k\pi-\delta,\frac{\pi}{2}+k\pi+\delta)$. Hence the cosine term $\abs{\cos((n+\frac{(\beta+1)}{2})\theta-\frac{\pi}{4})}$ can be estimated from below by $\abs{\cos(\frac{\pi}{2}+\delta)}$ for at least every second $n\in\N_0$ for fixed $\beta\in\N_0$. Now by the error estimate \eqref{3dim} and the asymptotic formula for the Bessel function (see \cite{Szego}, equation (1.71.7))
\begin{eqnarray*}
J_\mu(z)=\sqrt{\frac{2}{\pi z}} \cos(z-\frac{\mu\pi}{2}-\frac{\pi}{4})+\mathcal{O}(z^{-\frac{3}{2}}), \quad z\to\infty,
\end{eqnarray*}
we can conclude that there are $M\in\N$ and $C_1>0$, such that for every fixed $\beta\in\N_0$ we have for at least every second $n\in\N$ such that $N=n+(\beta+1)/2\ge M$, the estimate
\begin{eqnarray}
\left|\sqrt{ \frac{\theta}{\sin\theta}}J_0(N\theta)+\delta_1(N,\theta)\right|&\ge& \sqrt{\frac{2\theta}{\sin (\theta)\pi N\theta}}\abs{\cos(N\theta-\frac{\pi}{4})}- \frac{\Lambda}{N} \left(|J_0(N\theta)|+|J_1(N\theta)|\right)- C_1 (N\theta)^{-\frac{3}{2}}\notag\\
&\ge& (1-\epsilon')\sqrt{\frac{2\theta}{\sin (\theta)\pi N\theta}}\left|\cos(\frac{\pi}{2}+\delta)\right|, \quad \epsilon'\in(0,1).\label{d2}
\end{eqnarray}
Therefore equations \eqref{d1} and \eqref{d2} yield for some $\epsilon>0$
\begin{eqnarray*}
\abs{\frac{\phi_3}{\phi_1\phi_2}}^{n} \left|P^{(0,\beta)}_n\left(-\frac{2\phi_1\phi_2}{\phi_3}-1\right)\right|\ge \epsilon, 
\end{eqnarray*}
for at least every second $n\ge M$ and for every $\beta\in\N_0$.
Using these estimations it follows that for some $\epsilon>0$
\begin{eqnarray}
2\abs{A_x^\beta}\ge\left|\left\{n\in\N_0, n+\frac{\beta+1}{2}>M:\abs{\phi_1}^{-(n+\beta)} \abs{\phi_2}^{-n} \le x \epsilon \right\}\right|-1. \label{eqq}
\end{eqnarray}
Altogether we conclude (the estimation in \eqref{eqq} is only done for at most all $\beta\in\N_0$ fulfilling $\beta< -\log(x\epsilon)/\log\abs{\phi_1}$, $x>\epsilon^{-1}$, thus finitely many times)
$$f(x) \ge  \frac{1}{2}\left|\left\{(k_1,k_2)\in\N_0^2, k_1\ge k_2:\abs{\phi_1}^{-k_1} \abs{\phi_2}^{-k_2} \le x \epsilon \right\}\right|-CM^2- C'\log(x\epsilon),\quad x>\epsilon^{-1},$$
 for some $C,C'>0$ and the statement of the lemma follows easily.
\end{proof}

\begin{theorem}
Let $(\phi_1,\phi_2,\phi_3)\neq (0,0,0)$. Then the spatial autoregressive model \eqref{model} admits a causal solution $(Y_{t_1,t_2})_{(t_1,t_2)\in\Z^2}$ if and only if
\begin{enumerate}
\item[(i)] the polynomial $\Phi(z_1,z_1)=1-\phi_1 z_1-\phi_2 z_2-\phi_3 z_1 z_2$ has no zero on $\overline\D^2$, \quad and
\item[(ii)] if at least two coefficients of $\phi_1,\phi_2$ and $\phi_3$ are not equal to zero, then $\ew \log^2_+\abs{Z_{\textbf{0}}}<\infty$, otherwise $\ew \log_+\abs{Z_{\textbf{0}}}<\infty$.
\end{enumerate}
If those conditions hold and $\PP_{Z_{\textbf{0}}}$ is symmetric, then the unique symmetric causal solution is given by $Y_{\tbf}=\sum_{\kbf\in\N_0^2} \psi_{\kbf} Z_{\tbf-\kbf}, \tbf\in\Z^2$, where
\begin{equation}
\psi_{\kbf}=\sum_{j=0}^{k_1}\binom{k_2}{j}\binom{k_1+k_2-j}{k_2} \phi_1^{k_1-j}\phi_2^{k_2-j}\phi_3^{j},\quad \kbf=(k_1,k_2)\in\N_0^2.\label{coefflast}
\end{equation}
\end{theorem}
\begin{proof}
Suppose first that at least two coefficients of $\phi_1,\phi_2,\phi_3$ are not equal to zero. Then sufficiency of conditions $(i)$ and $(ii)$ as well as the representation \eqref{coefflast} follow from Theorem \ref{propwrong}, since
$$\Phi^{-1}(z_1,z_2)=\sum_{\kbf\in\N_0^d}  \psi_\kbf \zbf^\kbf,\quad \zbf\in\overline{\D}^2.$$
Conversely, suppose that $(Y_\tbf)_{\tbf\in\Z^2}$ is a causal solution of \eqref{model}. The necessity of condition $(i)$ follows by Corollary \ref{2dim}. It remains to show the necessity of condition $(ii)$. As in the proof of Theorem \ref{theorem2} we define the symmetrizations $(\tilde Y_\tbf)_{\tbf\in\Z^2}$ and $(\tilde Z_\tbf)_{\tbf\in\Z^2}$. Then $(\tilde Y_\tbf)_{\tbf\in\Z^2}$ admits the representation $\tilde Y_{\tbf}=\sum_{\kbf\in\N_0^2} \psi_{\kbf} \tilde Z_{\tbf-\kbf}$, where
$$\psi_{\kbf}=\sum_{j=0}^{k_1}\binom{k_2}{j}\binom{k_1+k_2-j}{k_2} \phi_1^{k_1-j}\phi_2^{k_2-j}\phi_3^{j},\quad \kbf=(k_1,k_2)\in\N_0^2,$$
and the convergence is almost surely in the rectangular sense. By Theorem \ref{theoremklesov2}  and Corollary \ref{thmklesov} we can conclude
\begin{equation}
\sum_{\kbf\in\N^2_0} \PP\left(\left|\psi_\kbf \tilde Z_{\tbf-\kbf}\right|>\epsilon\right)<\infty,\quad \forall \epsilon >0.\label{eq10}
\end{equation} 
But equation \eqref{eq10} implies finite second log-moment, because of the following estimates
\begin{equation}
\ew \log_+^2\abs{\tilde Z_{\textbf 0}}\le\sum_{k=1}^{\infty}  \log^2(k+1) \PP\left(k\le \abs{\tilde Z_{\textbf 0}}< k+1\right)\notag\le\sum_{k=1}^{\infty}  \left(\log(k)+1\right)^2 \PP\left(k\le \abs{\tilde Z_{\textbf 0}}< k+1\right).\label{lasteq}
\end{equation}
The last series is finite, if and only if $\sum_{k=1}^{\infty}  \log^2(k) \PP\left(k\le \abs{\tilde Z_{\textbf 0}}< k+1\right)$ converges.
By Lemma \ref{Last} the latter series can be estimated from above by
\begin{align*}
\sum_{k=1}^{\infty}  \log^2(k) \PP\left(k\le \abs{\tilde Z_{\textbf 0}}< k+1\right)
\le& C  \sum_{k=1}^{\infty} \left|\{\kbf\in\Z^2: \abs{\psi_{\kbf}}^{-1}\le k\}\right| \PP\left(k\le \abs{\tilde Z_{\textbf 0}}< k+1\right)\\
=& C  \sum_{k=1}^{\infty} \left|\{\kbf\in\Z^2: \abs{\psi_{\kbf}}^{-1}\in(k-1,k]\}\right|\PP(\abs{\tilde Z_{\textbf 0}}\ge k)\\
\le&\sum_{\kbf\in\N^2_0} \PP\left(\abs{\psi_\kbf \tilde Z_{\tbf-\kbf}}\ge 1\right)<\infty,\notag
\end{align*}
where the last inequality follows by equation \eqref{eq10}. Notice that the finite second log-moment of the symmetrization $\tilde Z_{\textbf 0}$ implies finiteness of the second log-moment of $Z_{\textbf 0}$.\bigbreak
It remains to consider the case, when two coefficients are equal to zero. In this case the model reduces to a one dimensional model. That is also the case, if $\phi_1=\phi_2=0$ by defining the operator $\tilde B=B_1 B_2$. The model lives only on the diagonal $\tbf=(t,t)\in\Z^2, t\in\Z$. Thus, the results \cite{Brockwelllindner} for the time series model can be applied, which yields necessity and sufficiency of $\ew \log_+\abs{Z_0}<\infty$ and condition $(i)$ in that case.
\end{proof}\\
An equivalent condition for $(i)$ describing the parameter regions is given in Proposition 1 of Basu and Reinsel \cite{Basu}.


\section{Acknowledgement}
We would like to thank Professor Oleg Klesov for a discussion on this topic and for pointing out the reference \cite{Klesov} to us.



\begin{thebibliography}{nn98}
\normalsize

\bibitem{Basu}
Basu, S. and Reinsel, G. (1993).
Properties of the spatial unilateral first-order ARMA model.
\newblock Advances in Applied Probability \textbf{25} 631-648.

\bibitem{Besag}
Besag, J. E. (1972).
On the correlation structure of some two-dimensional stationary processes.
\newblock Biometrika \textbf{59} 43-48.

\bibitem{Bocher}
Bôcher, M. (1964). 
Introduction to Higher Algebra.
\newblock Dover publications, 2nd edition.

\bibitem{Brockwelldavis}
Brockwell, P. J. and Davis, R. A. (1991). Time Series: Theory and Methods.
\newblock Springer, 2nd edition.

\bibitem{Brockwelllindner}
Brockwell, P.J. and Lindner, A. (2010).
Strictly stationary solutions of autoregressive moving average equations.
\newblock Biometrika \textbf{97} 765-772.

\bibitem{Chow}
Chow, Y. S. and Teicher, H. (1997).
Probability Theory - Independence, Interchangeability, Martingales.
\newblock Springer, 3rd edition.

\bibitem{Fray}
Fray, R. D. and Roselle, D. P. (1971).
Weighted lattice paths.
\newblock Pacific Journal of Mathematics, Vol. \textbf{37} 85–96.

\bibitem{Hetyei}
Hetyei, G. (2009).
Shifted Jacobi polynomials and Delannoy numbers.
\newblock ArXiv e-prints  {\tt arXiv:0909.5512v2}.

\bibitem{Kallenberg}
Kallenberg, O. (2002). 
Foundations of Modern Probability.
\newblock Springer, 2nd edition.

\bibitem{Klesov}
Klesov, O. (1995).
Almost sure convergence of multiple series of independent random variables.
\newblock Theory of Probability and its Applications \textbf{40} 52-65.

\bibitem{Klesov2}
Klesov, O. (1980).
Three-series theorems for random fields with independent values.
\newblock Ser. Math. and Meth. \textbf{22} 35-40 (in Russian).

\bibitem{Range}
Range, M. (1986).
Holomorphic Functions and Integral Representations in Several Complex Variables.
\newblock Springer, 1st edition, Berlin.

\bibitem{Rosenblatt2}
Rosenblatt, M. (1985). 
Stationary Sequences and Random Fields.
\newblock Birkhäuser, 1st edition.

\bibitem{Rosenblatt}
Rosenblatt, M. (2000). 
Gaussian and Non-Gaussian Linear Time Series and Random Fields.
\newblock Springer, Berlin, 1st edition.

\bibitem{Rudin}
Rudin, W. (1969). 
Function Theory in Polydiscs.
\newblock Benjamin, New York.

\bibitem{Shapiro}
Shapiro, V. L. (2011).
Fourier Series in Several Variables with Applications to Partial Differential Equations.
\newblock Taylor and Francis, 1st edition.

\bibitem{Stein}
Stein, L.E. and Weiss, G. (1971).
Introduction to Fourier Analysis on Euclidean Spaces.
\newblock Princeton University Press, Kassel.

\bibitem{Szego}
Szegö, G. (1967).
Orthogonal Polynomials.
\newblock American Mathematical Society, 3rd edition.

\bibitem{Tjostheim}
Tjøstheim, D. (1978). 
Statistical spatial series modelling.
\newblock Advances in Applied Probability \textbf{10} 130-154.

\bibitem{Tyrtyshnikov}
Tyrtyshnikov, E. (1997).
A Brief Introduction to Numerical Analysis.
Birkhäuser.

\bibitem{Vollenbroeker}
Vollenbröker, B. (2012).
Strictly stationary solutions of ARMA equations with fractional noise.
\newblock Journal of Time Series Analysis \textbf{33} 570-582.

\bibitem{Whittle}
Whittle, P. (1954). 
On stationary processes in the plane.
\newblock Biometrika \textbf{41} 434-449.

\bibitem{Wong}
Wong, R. and Zhao, Y.-Q.  (2003). 
Estimates for the error term in a uniform asymptotic expansion of the Jacobi polynomials.
\newblock Analysis and Applications \textbf{1} 213-241.

\bibitem{Wong2}
Wong, R. and Zhao, Y.-Q.  (2004). 
Uniform asymptotic expansion  of the Jacobi polynomials in a complex domain.
\newblock Proc. R. Soc. Lond. A \textbf{460} 2569-2586.

\end{thebibliography}
\end{document}